\documentclass[11pt,a4paper]{amsart}
\usepackage[utf8]{inputenc}

\usepackage{amsmath}
\usepackage{amsfonts}
\usepackage{amssymb}
\usepackage{amsthm}

\usepackage{mathrsfs}

\usepackage[left=2.54cm,right=2.45cm,top=2.54cm,bottom=2.54cm]{geometry}

\usepackage{hyperref}
\usepackage{mathtools}
\mathtoolsset{showonlyrefs}

\usepackage{enumerate}

\theoremstyle{plain}
\newtheorem{thm}{Theorem}
\newtheorem{lemma}[thm]{Lemma}
\newtheorem{prop}[thm]{Proposition}
\newtheorem{cor}[thm]{Corollary}

\theoremstyle{remark}
\newtheorem{rmk}{Remark}
\newtheorem{hyp}{Assumption}

\theoremstyle{definition}

\newtheorem{definition}{Definition}

\usepackage{bbm}

\usepackage{xcolor}

\newcommand{\NN}{\mathbb{N}}
\newcommand{\M}{\mathbb{M}}

\newcommand{\R}{\mathbb{R}}

\newcommand{\TT}{\mathbb{T}}

\newcommand{\e}{\epsilon}

\newcommand{\mc}{\mathcal}
\newcommand{\mb}{\mathbb}

\newcommand{\dom}{{\M \times \R^d}}

\newcommand{\be}{\begin{equation}}
\newcommand{\ee}{\end{equation}}

\newcommand*\dd{\mathop{}\!\mathrm{d}}

\renewcommand{\div}{\operatorname{div}}

\numberwithin{thm}{section}
\numberwithin{equation}{section}

\title{Kinetic-Type Mean Field Games with Non-separable local Hamiltonians}
\author{David M. Ambrose}
\address{Drexel University, Department of Mathematics, Philadelphia, PA 19104 USA}
\email{dma68@drexel.edu}
\author{Megan Griffin-Pickering}
\address{Department of Mathematics, University College London, London, UK \& Heilbronn Institute for Mathematical Research, Bristol, UK}
\email{m.griffin-pickering@ucl.ac.uk}
\author{Alp\'ar R. M\'esz\'aros}
\address{Department of Mathematical Sciences, Durham University, Durham, UK}
\email{alpar.r.meszaros@durham.ac.uk}

\begin{document}

\begin{abstract}
We prove well-posedness of a class of {\it kinetic-type} Mean Field Games, which typically arise when agents control their acceleration. Such systems include independent variables representing the spatial position as well as velocity. We consider non-separable Hamiltonians without any structural conditions, which depend locally on the density variable.  Our analysis is based on two main ingredients: an energy method for the forward-backward system in  Sobolev spaces, on the one hand and on a suitable {\it vector field method} to control derivatives with respect to the velocity variable, on the other hand. The careful combination of these two techniques reveals interesting phenomena applicable for Mean Field Games involving general classes of drift-diffusion operators and nonlinearities.
While many prior existence theories for general mean field games systems take the final datum function to be smoothing, we can allow this function to be non-smoothing, i.e. also depending locally on the final measure.  Our well-posedness results hold under an
appropriate smallness condition, assumed jointly on the data. 
\end{abstract}

\maketitle

\section{Introduction}

Mean Field Games (MFGs) models were first proposed around the mid 2000s in the seminal works of Lasry--Lions (\cite{lasryLions2,lasryLions3}) and Huang--Malham\'e--Caines (\cite{HuaMalCai}). The objective of both groups was to characterize limits of Nash equilibria for stochastic $N$-player games, when $N\to+\infty$. A natural way to analyze MFGs is via a coupled system of forward-backward nonlinear PDEs: a Hamilton--Jacobi--Bellman (HJB) equation describing the evolution of the value function of a representative agent and a Kolmogorov--Fokker--Planck (KFP) equation describing the evolution of the agent population density.

In this paper we consider the following {\it kinetic-type} MFG system:
\be \label{eq:MFG}
\left\{
\begin{array}{l}
- \partial_t u(t,z) - v \cdot D_x u(t,z) -  \frac{1}{2} \Delta_v u(t,z) - \e H(t,z, m(t,z), D_v u(t,z)) = 0,\\[5pt] 
\partial_t m(t,z)  + v \cdot D_x m(t,z) -  \frac{1}{2} \Delta_v m(t,z) + \e \div_v \left ( m D_p H(t,z, m(t,z), D_v u(t,z)) \right ) = 0,\\ [5pt]
u(T,z) = \delta g(z,m_T(z)),\\[5pt] 
m(0,z) = m^0(z), 
\end{array}
\right.
\ee
where the two PDEs are posed over the domain $(0,T) \times \dom$, where $\M$ is either the $d$-dimensional flat torus $\TT^d$ or the whole $d$-dimensional Euclidean space $\R^d$, and $T>0$ stands for a given time horizon.  Throughout the paper we use the short notation $z := (x,v) \in \dom$ to denote the position and velocity of the agents.

The data consist of $m^{0}\in\mathscr{P}(\dom)\cap H^{s}(\dom)$ which stands for the initial agent distribution (a probability measure, whose density is taken in a suitable Sobolev space), a Hamiltonian $H:(0,T)\times\dom\times[0,+\infty)\times\R^{d}\to\R$ and a final cost function $g:\dom\times[0,+\infty)\to\R$. The penultimate variable in the domain of definition of $H$ is the placeholder for the density variable and the last variable stands for the placeholder of the gradient of the value function with respect to $v$. Similarly, the last variable in the domain of definition for $g$ is a placeholder for $m(T,x,v)$. We emphasize at this point that both $H$ and $g$ are considered to depend {\it locally} on the density variable $m$, i.e. we always evaluate them at $m(t,x,v)$. We have placed two parameters $\e,\delta>0$ in front of $H$ and $g$, respectively to emphasize that the `size' of the Hamiltonian and final cost function can be tuned. In particular, in this paper we seek to prove a local well-posedness result for \eqref{eq:MFG} in suitable Sobolev spaces $H^{s}(\dom)$ and we will achieve these under suitable {\it smallness conditions} on these data and the additional parameters.

\medskip

\noindent {\bf Literature review.} Most of the well-posedness theories for MFGs assume a structural condition on the Hamiltonian, specifically that $H$ is additively separable: in the current context, the separability assumption would look like 
$$H(t,x,v,D_{v}u)=H_{0}(t,x,v,D_{v}u)-F(t,x,v,m),$$  
for some $H_{0}$ and $F$. The first  of these, $H_{0},$ would again be known as the Hamiltonian, and the second of these, $F,$ would be known as the coupling function or running cost. Frequently the Hamiltonian $H_{0}$ is then assumed to be convex in the variable which stands for the generalized momentum and the coupling function is assumed to be monotone.

Examples of well-posedness theories under these separability assumptions for the MFG PDE system (in the usual, non-kinetic setting, i.e. without the presence of the velocity variables, and when Hamiltonians depend locally on the density variable) are \cite{ambroseMFG1,jamesonDegenerate,CirGof:20,CirGof:21,gomesLog,gomesSub,gomesSuper,lasryLions2,lasryLions3,MesSil:18,porrettaARMA}. For a relatively complete literature review on these models we refer to the lecture notes \cite{CarPor:20}. 

Many important models of MFGs arising in applications in fact involve non-separable Hamiltonians, such as the model of household wealth and savings \cite{achdouMacroeconomic}, models for natural resource extraction \cite{chanSircar}, and
models of traffic flow with congestion effects \cite{AchPor, GhaMas:1, congestion2}.
There are many fewer well-posedness theories for MFGs involving non-separable Hamiltonians; without the structure afforded by these assumptions, smallness assumptions are typically relied upon to prove existence and uniqueness results. The first author proved existence of solutions for MFGs with data in spaces related to the Wiener algebra in \cite{ambroseMFG2} and with data in Sobolev spaces \cite{ambroseMFG4} (see also \cite{AmbMes:23}), and the authors of \cite{cirantGianniMannucci} have proved similar results.  Other than these general existence theorems, there are also existence theorems for mean field games with non-separable Hamiltonians for specific applications \cite{AchPor, ambroseMFG3, GhaMas:1, congestion2,graber1}. When Hamiltonians and final data functions are non-local and regularizing in the measure variable, such well-posedness theories have been settled for short time (cf. \cite{CarDel,CarCirPor}). The only global in time well-posedness results involving non-separable nonlocal Hamiltonians are in the displacement monotone setting, obtained in \cite{GanMesMouZha, MesMou:24, BanMesMou}.

The present work is the first to provide existence theory for kinetic-type mean field games with non-separable Hamiltonians that depend locally on the measure variable. In contrast to many results for MFGs in the literature, we will consider non-smoothing final data functions, $g.$
The work \cite{cirantGianniMannucci} on existence theory for non-separable MFGs (with local Hamiltonians) used smoothing final data
functions and alluded to the need for an additional smallness condition if this functions were instead non-smoothing.
In \cite{ambroseMFG4}, through the bulk of the work smoothing final cost functions were used, but this additional smallness
condition was specified in remarks at the end.  In the present work, we use from the beginning a non-smoothing $g$
and deal clearly with the smallness conditions which are therefore needed for the argument.

\medskip

Kinetic-type MFG have been considered relatively recently in the literature. These models are prototypical examples for MFGs where the agents are subject to general control systems in their individual optimization problems. In particular, such models have been first considered in the setting when agents control their acceleration. In the setting of regularizing separable Hamiltonians and final cost functions, such models (with control on the agents' acceleration or in the case of more general linear control systems) have been intensively studied recently in \cite{AchManMarTch:20, AchManMarTch:22, BarCar, CanMen:20, accelerationInfinity, DraFel, FelGomTad,GiaSil, ManMarMarTch, ManMarTch, acceleration3}. 

To the best of our knowledge, kinetic-type MFGs models involving Hamiltonians depending locally on the measure variable, have been considered in only two papers in the past. In \cite{Griffin-Pickering-Meszaros}, the second and third authors proposed a variational framework to deal with first order models, while in \cite{Mim} the author studied weak and renormalized solutions in the case of separable Hamiltonians.

\medskip

\noindent {\bf Our contributions in this paper.} 
The main theorem of this paper can be formulated in an informal way as follows (we refer to Theorem \ref{thm:main_proved} for the precise statement).

\begin{thm}\label{mainTheorem}
Let $H$ and $g$ satisfy suitable regularity assumptions. Let $m^{0}\in H^{s}_z(\dom)$ be given, where $s$ is an integer such that $s>d+1$. Then there exist $\e_\ast>0, \delta_\ast>0, T_\ast>0\vphantom{, K_\ast}$ such that, for all $\e\in(0, \e_\ast), \delta\in(0,\delta_\ast), 
T\in (0,T_\ast), \vphantom{K_\ast}$ the Mean Field Games system \eqref{eq:MFG} has a unique solution $(u,m) \in (C^0_t{([0,T];H^s_z(\dom))})^2$ for which $(D_v u , D_v m) \in ( L^2_t{([0,T];H^s_z(\dom;{\R^{d}}))})^2$. If in addition $s>d+2$, this solution is classical.

The parameters $(\e_\ast, \delta_\ast, T_\ast\vphantom{, K_\ast})$ can be made explicit given $H,$ $g,$ and $\| m^{0} \|_{H^s_z}$.
\end{thm}

To prove Theorem \ref{mainTheorem} we proceed along a somewhat similar program as for the proof of existence and uniqueness of solutions to non-kinetic non-separable mean field games in \cite{ambroseMFG4}. However, along the way we need to develop new techniques to account for additional difficulties.

A central issue is the differing roles of derivatives in the $v$-variables versus those with respect to the $x$-variables: the diffusion is degenerate, acting in the $v$ directions only, while $x$ derivatives $D_x$ appear only through the transport operator $\partial_t + v \cdot D_x$. 
In order to obtain suitable energy estimates on solutions we therefore require new methods compared to \cite{ambroseMFG4}.
The specific technique that we use is to measure $H^s$ regularity in $z$ using a modified norm, which varies over time $t$ in a manner adapted to the transport---see Subsection~\ref{sec:transport}. 
The use of such modified quantities (variously described depending on context as energies, entropies, or Lyapunov functionals) is long established and ubiquitously used in the literature on kinetic equations to obtain well-posedness results and regularization and decay estimates (see e.g. \cite{Kawashima, Villani, HerauFP, HerauNier}).
Our presentation is phrased in terms of a family of vector fields that commute with the transport operator, taking inspiration from the approach of \cite{Smulevici} to the study of small data solutions for 
for the Vlasov--Poisson system by an adaptation to the setting of transport equations of the vector field method for wave equations \cite{Klainerman}; see also
\cite{FJS, Luk, Chaturvedi, ChaturvediLukNguyen, Bigorgne, BCZD} for applications of vector field methods to stability/decay results for other kinetic models from mathematical physics.

The spatial domain in \cite{ambroseMFG4} was taken to be the torus, which is of course compact; this provided the possibility
of using compactness to prove convergence of an iterative scheme and construct solutions. In the present work, however, the velocity
variable $v$ cannot be taken to be from a compact domain. We therefore replace the compactness argument with a Banach fixed point argument; we use energy estimates to find a smallness condition such that a contraction mapping exists between suitably chosen function spaces.
Moreover, we are required to prove additional estimates as compared to \cite{ambroseMFG4}, to ensure that the solutions are controlled at large velocities.

\begin{rmk}
\begin{enumerate}
\item In this paper we consider Hamiltonians and final data functions that depend locally on the density variable. Our approach, however, under suitable similar assumptions would also result in a well-posedness theory in cases where the data involve nonlocal (or both local and nonlocal) dependence on the measure variable. See Subsection~\ref{sec:MFGEstimates} for further discussion on this point.
\item We chose the differential operator present in \eqref{eq:MFG}, as this arises in models where agents control their acceleration. However, the very same techniques that we develop here would apply to a more general class of MFGs systems.

For example, our approach can immediately be applied to models in which players' individual control system is a linear system of the form
\be \label{eq:SDEGeneral}
\dd Z_t = (B Z_t + \Pi \alpha_t ) \dd t + \sigma \dd W_t ,
\ee
where the state variable $Z_t \in \R^n$, $\alpha_t \in \R^n$ represents the control variable, $W$ is an $n$-dimensional Brownian motion, and
$B, \Pi, \sigma \in {\R^{n \times n}}$ are constant matrices, such that $\Pi$ is of rank $k$ (with possibly $k < n$), and $\sigma$ is such that $| \sigma^\top \xi |^2 \geq c_0 |\Pi^\top \xi|^2$ for some $c_0 > 0$ for all $\xi \in \R^n$. This corresponds to a MFG system of the form
\be
\begin{cases}
- \partial_t u - (Bz) \cdot D_z u - \frac{1}{2} \sigma \sigma^\top : D_z^2 u + H(t,z,m, \Pi^\top D_z u)  = 0, \\[5pt]
\partial_t m + \div_z ( Bz m ) - \frac{1}{2} D_z^2 : (\sigma \sigma^\top m ) + \div_z (m \Pi D_p H(t,z,m, \Pi^\top D_z u)  )  = 0 ,
\end{cases}
\ee
supplemented with suitable initial and final conditions.
It would also be possible to generalise the method to general drifts and non-constant diffusion matrices.
The key condition is that the diffusion term should be coercive in the same directions of $D_z$ as those upon which the Hamiltonian $H$ depends.

We emphasise that here we do not impose any hypoellipticity condition on the diffusion operator (i.e. (parabolic) H\"ormander/Kalman rank condition), although for the kinetic case $\partial_t + v \cdot D_x - \frac{1}{2} \Delta_v$ this is satisfied.
It has been known since the celebrated work of H\"ormander \cite{Hormander} that diffusion equations of the form 
\be \label{eq:diffusion}
\partial_t m - V_0 m - \frac{1}{2} \sum_{i=1}^l V_i^2 m = f, \qquad z \in \R^n,
\ee
where $V_0, \ldots, V_l$ are $C^1$ vector fields, enjoy smoothing properties under a non-degeneracy condition on the vector fields. Namely, defining
\be
\mc{V}_0 : = \left \{ V_1, \ldots, V_l \right \}, \quad \mc{V}_j : = \mc{V}_{j-1} \cup \left \{ [W, V_j] : j \geq 0, W \in \mc{V}_{j-1} \right \},
\ee
the diffusion equation \eqref{eq:diffusion} is hypoelliptic if
the vector fields $\bigcup_j \mc{V}_j$ span $\R^n$ at every point $z$ of the state space. 
In the constant coefficient case \eqref{eq:SDEGeneral}, this is equivalent to the \emph{Kalman rank condition} that the matrix $(\Pi, B \Pi, \ldots B^{n-1} \Pi)$ is of full rank $n$.

The gain of regularity in $m$ compared to $f$ varies depending on the direction in which derivatives are taken, according to the type of commutators required to produce that direction \cite{RothschildStein}. In particular, directions involving commutators with the drift field $V_0$ gain strictly less than one derivative.
For example, in the case of the operator $\partial_t + v \cdot D_x - \frac{1}{2}\Delta_v$ appearing in the MFG system \eqref{eq:MFG},
we would expect to have at most $H^{s + 1/3}_x$ estimates for a solution with $H^s_{x,v}$ data and $H^s_xH^{s-1}_v$ source \cite{FollandStein, RothschildStein, Bouchut, NiebelZacher}.
On the other hand, since the Hamiltonian in a Mean Field Game would not be expected to be regularising with respect to the momentum variable, the nonlinearities in the corresponding Mean Field Games system lose a full derivative $\Pi^\top D_z u$ on the value function in the control directions. This loss would not be compensated by the fractional gain of regularity arising from the hypoellipticity property. 
For this reason we impose the restriction that the control directions and diffusive directions should be aligned.

It is interesting to notice that in \cite{DraFel}, where H\"ormander-type operators with no drift $V_0 \equiv 0$ are considered, the authors similarly assume that the Hamiltonian depends on a subgradient given by the $\mc{V}_0$ directions. Although this philosophy is similar in spirit to ours, \cite{DraFel} used completely different techniques from ours, as the model problem there is time independent and the dependence of the Hamiltonian on the measure variable was assumed to be regularizing.
\item At no point in our analysis do we really use the fact that the Hamiltonian $H$ appearing in the HJB equation is necessarily related to the drift involving $D_{p}H$ in the KFP equation. Indeed, the very same techniques developed in this paper apply to systems of the form  
\be
\begin{cases}
- \partial_t u - (Bz) \cdot D_z u - \frac{1}{2} \sigma \sigma^\top : D_z^2 u + \mc{H} [m, \Pi^\top D_z u ]  = 0, \\
\partial_t m + \div_z ( Bz m ) - \frac{1}{2} D_z^2 : (\sigma \sigma^\top m ) + \div_z (\Pi \mc{J} [m, \Pi^\top D_z u ])  = 0 ,
\end{cases}
\ee
equipped with suitable initial and terminal conditions. Here $\mc{H}$ is a scalar valued operator, while $\mc{J}$ is a vector valued operator, acting between suitable function spaces, and satisfying our standing assumptions. Such systems should remind the reader of the so-called {\it extended mean field games}, introduced recently in \cite{LioSou} (see also \cite{Mun,MouZha}). Therefore, our main results hold true for a suitable class of extended MFG models, involving a general class of drift-diffusion operators described above. Interestingly, \cite{DraFel} also studied in fact an extended MFG system, as there also the Hamiltonian is not necessarily related to the drift operator appearing in the KFP equation.
\item We would like to emphasize once more that it is purely for pedagogical considerations that we restrict our study to systems of the form \eqref{eq:MFG}. It would be without too much philosophical difficulty (but at the price of a much heavier technical language) to extend our main results to more general systems described above. At some crucial places in our analysis below, we have emphasized what would change if one would consider a more general scenario, i.e. when involving state equations of the form \eqref{eq:SDEGeneral}.

\end{enumerate}
\end{rmk}

\medskip

The plan of the paper is as follows. In Section \ref{preliminarySection} we begin by listing some notation, our standing assumptions, and other preliminaries including definitions of function spaces and norms and useful estimates. These estimates include bounds related to the method of vector fields we use, as well as technical composition estimates in Sobolev spaces. These will serve as building blocks for the steps leading to the proof of our main theorem. In Section \ref{linearSection}, we provide a number of estimates for a prototypical linear degenerate diffusion equation (with source), which will be used later to define a map that we will employ in a fixed point argument. Specifically, in Subsection \ref{linearEquation} we provide estimates for solutions of a single initial value problem, and the corresponding terminal value problem. Then, in Subsection~\ref{sec:MFGEstimates}
we set up
a system of linear equations with source
whose solution map will form the basis of our fixed point argument. 
In Section~\ref{sec:FixedPoint} we show that, for a suitable range of the parameters $(\delta, \e, T)$ this map is indeed a contraction on a complete metric space, proving our main theorem,  Theorem \ref{mainTheorem}, on existence and uniqueness of solutions of \eqref{eq:MFG}.  Finally, in Appendix \ref{appendixSection} we give the technical proof of an important estimate.

\section{Standing assumptions and preliminaries}\label{preliminarySection}

We use the notation $s_\ast := d+1$ to denote the critical (integer) exponent $s_\ast$ such that 
$H^s_z(\dom)$ embeds continuously into $L^\infty(\dom)$ for all $s \geq s_\ast$.

We impose the following standing assumptions throughout the paper.

\begin{hyp} \label{hyp:main}
For some integer $s \geq s_\ast + 1$,
$H$ is a continuous function of all arguments, and a $C^{s+2}$ function of $(z,m,p)$.
There exists a non-decreasing function $ \Phi_H:[0,+\infty)\to[0,+\infty)$ such that, for any $r > 0$ and $t \in [0,T]$,
\be \label{eq:hyp:main-unif}
\| H(t, \cdot, \cdot , \cdot ) \|_{C^{s+2} (\dom \times B_r(0))} \leq \Phi_H(r)
\ee
and, for  all multi-indices $\alpha$ with $|\alpha| \leq s$ and any functions $(m,p) \in (L^2\cap L^\infty)(\dom ; \R \times \R^{d})$,
\be \label{eq:hyp:main-L2}
\sup_{t \in [0,T]} \| \partial^\alpha_z H(t, \cdot, m, p) \|_{L^2_z} \leq \Phi_H \left ( \| (m,p) \|_{L^2_z \cap L^\infty_z} \right ) .
\ee
\end{hyp}

We also assume the following on the terminal coupling $g$.

\begin{hyp} \label{hyp:terminal}
For the same $s\ge s_\ast+1$ as above, $g : \dom \times [0,+\infty) \to \R$ is a $C^{s+1}$ function of all arguments, for which there exists a non-decreasing function $\Phi_g : [0,+\infty) \to [0,+\infty)$ such that, for any $r > 0$,
\be
\| g \|_{C^{s+1}(\dom \times B_r(0))} \leq \Phi_g (r),
\ee
and for  all multi-indices $\alpha$ with $|\alpha| \leq s$ and any function $m \in (L^2\cap L^\infty)(\dom)$,
\be \label{eq:hyp:g-L2}
\| \partial_{z}^\alpha g(\cdot, m) \|_{L^2_z} \leq \Phi_g \left ( \| m \|_{L^2_{z} \cap L_{z}^\infty} \right ) .
\ee

\end{hyp}

{

\begin{rmk}
The $L^2$ type condition \eqref{eq:hyp:main-L2} is satisfied for example if 
\be
 | \partial^\alpha_z H(t, z, m, p) |  \lesssim | (m, p)| \; \text{as} \; (m, p) \to 0  .
\ee
Since $H$ is $C^{s+2}$ with the uniform-in-$z$ bounds \eqref{eq:hyp:main-unif}, this is equivalent to asking that 
\be
 | \partial^\alpha_z H(t, z, 0, 0)  |  \equiv 0 ,
\ee
which reduces to 
\be \label{hyp:H0}
H(t,z, 0, 0) \equiv 0 .
\ee

Similarly, the condition 
\be \label{hyp:g0}
g(z, 0) \equiv 0, 
\ee
suffices to ensure that \eqref{eq:hyp:g-L2} is satisfied.

Moreover, it suffices for conditions \eqref{hyp:H0}-\eqref{hyp:g0} to hold only outside of a compact set in $z$.
\end{rmk}

\begin{rmk} \label{rmk:examples}
Examples of admissible Hamiltonians include functions of the form $H(t,z, m, p) = a(t,z)f(m, p)$, where $a \in C^0_t \left ( C^{s+2}_z \cap H^s_z \right )$, and $f$ is any $C^{s+2}$ function. Functions $H$ depending on $m$ and $p$ only and satisfying \eqref{hyp:H0} are also included.
For example, we would like to be able to include
\begin{equation}\label{almostExample}
H(m,p) = \frac{|p|^2}{a + m}, \qquad \text{for any} \; a > 0 .
\end{equation}
However until we know that $m>-a,$ this is not sufficiently smooth.
We therefore first let $\phi_{a}$ be a $C^{\infty}$ function such that
$\phi_{a}(\omega)=\omega$ for all $\omega\geq0,$ and $\phi_{a}(\omega)>-\frac{a}{2}$ for all $\omega.$
We then modify \eqref{almostExample} to
\begin{equation}\label{cutoffExample}
H(m,p) = \frac{|p|^{2}}{a+\phi_{a}(m)}.
\end{equation}
This now satisfies the assumptions.  Furthermore, if the solution $m$ which we prove to exist is in fact a probability distribution (so that $m\geq0$),
then \eqref{cutoffExample} reduces to \eqref{almostExample}.  As we will comment in  Remark \ref{probabilityRemark} below, it is indeed the case that
if $m_{0}$ is a probability distribution, then the solution we prove to exist will remain a probability distribution at later times.
\end{rmk}
}

In the rest of this section, we will give results on vector fields, function spaces, and basic estimates which will be useful many times throughout the remainder of the paper.

\subsection{Function spaces}
We make use of the usual $L^p$ spaces over $[0,T] \times \dom$, with norm
\be
\| f \|_{L^p}  = \left ( \int_0^T \int_{\dom} |f(t,x,v)|^p \dd x \dd v \dd t \right )^{1/p}.
\ee
Subscripts indicate $L^p$ norms taken with respect to only some of the variables, e.g.
\be
\| f(t, \cdot) \|_{L^p_z}  = \left (  \int_{\dom} |f(t,x,v)|^p \dd x \dd v \right )^{1/p} ,
\ee
where we have used the shorthand $z = (x, v) \in \dom$, as well as mixed $L^p$ norms, with different orders of integration in different variables. For example,
\be
\| f \|_{L^\infty_t L^2_z} : = \sup_{t \in [0,T]} \left ( \int_{\dom} |f(t,x,v)|^2 \dd x \dd v \right )^{1/2} .
\ee

We denote by $\dot H^s_z$ the homogeneous Sobolev spaces of functions having all $s$-order derivatives with respect to the variable $z$ in $L^2$, equipped with the semi-norm $\| \cdot \|_{\dot H^s}$ defined as follows:
\be
\| f \|_{\dot H^s_z}^2 : = \sum_{|\beta| = s} \| \partial^\beta_z f \|_{L^2_z}^2 .
\ee
The non-homogeneous Sobolev spaces of $L^2$ functions with all $z$-derivatives of order less than or equal to $s$ will be denoted by $H^s_z$. We will use the norm $\| \cdot \|_{H^s_z}$ defined by
\be
\| f \|_{ H^s_z}^2 : = \| f \|_{L^2_z}^2 + \| f \|_{\dot H^s_z}^2,
\ee
which by interpolation is equivalent to the usual one involving a sum over derivatives of all admissible orders.

Similarly, $\dot H^s_x$, $H^s_v$ etc. will denote the corresponding Sobolev spaces where derivatives are taken with respect to, respectively, the $x$ or $v$ variables only.

We will use similar shorthand notations for mixed spaces, for example 
\be
L^2_t H^s_z : = L^2_t ((0,T); H^s_z(\dom)).
\ee 
For functions with values in spaces of dimension higher than one, we use abbreviations of the form $(L^2_t H^s_z)^l : = L^2_t ((0,T); H^s_z(\dom ; \R^l)) $ (here $l \in \NN$).

When there is no ambiguity involved, especially for functions which do not depend on the time variable, we use interchangeably the notation $L^p$ and $L^p_z=L^p_{(x,v)}$, and similarly in the case of the corresponding Sobolev spaces.

Let $X$, $Y$ be normed spaces. The notation $X \cap Y$ denotes the intersection space (see e.g. \cite{Bergh-Lofstrom}), with the norm $\| \cdot \|_{X \cap Y}$ being defined by
\be
\| f \|_{X \cap Y} = \| f \|_X + \| f \|_Y .
\ee

The notation $\iota_{\Omega}: \Omega \to \Omega$ always denotes the identity map $\iota_\Omega(\omega) = \omega, \; \forall \omega \in \Omega$, for any set $\Omega$. When there is no ambiguity, we drop the subscript $\Omega$ in $\iota_\Omega$.

\subsection{Tools for the Transport Operator} \label{sec:transport}

In this subsection, we will introduce the key tools used in our analysis to handle the kinetic transport operator $\partial_t + v \cdot D_x$: a basis of \emph{vector fields}, chosen so as to commute with $\partial_t + v \cdot D_x$, and the \emph{flow map} describing the transport.

\begin{definition}[Vector fields]
For $j = 1, \ldots d$, let
\be
\gamma_j = \partial_{x_j} .
\ee
For $j = d+1, \ldots, 2d$, let
\be
\gamma_j = t \partial_{x_{j-d}} + \partial_{v_{j-d}} .
\ee 
In other words, the column vector $\gamma$ with entries $(\gamma_j)_{j=1}^{2d}$ can be expressed in the form
\be
\gamma = 
\begin{pmatrix}
I_d & 0_d \\
t I_d & I_d
\end{pmatrix}
D_z,
\ee
where $I_d$ denotes the identity matrix in dimension $d$, and $0_d$ is the zero matrix in dimension $d$. {Sometimes we write explicitly the dependence of $\gamma$ on $t$, as $\gamma(t)$.}
\end{definition}

\begin{rmk}
In the case of the general constant coefficient control system \eqref{eq:SDEGeneral}, one would consider the differential operators given by the entries of the column vector $e^{tB^\top} D_z$, where $e^{tB^\top}$ is the matrix such that
\be
\frac{\dd}{\dd t} e^{tB^\top} = B^\top e^{tB^\top}, \qquad e^{tB^\top} \vert_{t=0} = I .
\ee
\end{rmk}

\begin{rmk} \label{rmk:commute}
Observe that the following commutation relations hold. All $\gamma_j$ commute with each other:
\be \label{eq:gamma-commutes}
\left[ \gamma_j, \gamma_k \right ] = 0 \qquad \text{for all} \; j,k .
\ee
Moreover, they commute with pure $v$-derivatives, and hence in particular with the Laplacian $\Delta_v$:
\be
\left[ \gamma_j, \partial_{v_k} \right ], = 0 \qquad \left[ \gamma_j, \Delta_{v} \right ] = 0, \qquad \text{for all} \; j,k . 
\ee
Finally, each $\gamma_j$ commutes with the transport operator:
\be
\left[ \gamma_j, \partial_t + v \cdot D_x \right ] = 0 \qquad \text{for all} \; j .
\ee
\end{rmk}

We will use multi-index notation to denote compositions of the $\{ \gamma_j \}_{j=1}^{2d}$, in the usual way: for a multi-index $\beta \in \NN^{2d}$, let
\be \label{def:gamma-beta}
\gamma^\beta = \gamma_1^{\beta_1} \ldots \gamma^{\beta_{2d}}_{2d}
\ee
Since the operators $\{ \gamma_j \}_{j=1}^{2d}$ commute \eqref{eq:gamma-commutes}, all finite compositions thereof can be expressed in the form \eqref{def:gamma-beta}.

We note the following relations, which we will use to convert between Sobolev estimates and estimates in terms of $\gamma^\beta$.

\begin{lemma} \label{lem:VF-conversion}
For any $s \in \mb{N}$ there exists $C_{s,d}>0$ such that for all $f\in \dot H^s_{z}(\dom)$ we have the following two inequalities
\be
\| f \|_{\dot H^s_{z}} \leq C_{s,d} (1+t)^{s} \left ( \sum_{|\beta| = s} \| \gamma^\beta f \|^2_{L^2_{z}} \right )^{1/2}, \quad \left ( \sum_{|\beta| = s} \| \gamma^\beta f \|_{L^2_{z}}^2 \right )^{1/2} \leq C_{s,d} (1+t)^{s} \| f \|_{\dot H^s_{z}} .
\ee

\end{lemma}
\begin{proof}

Writing
\be
A(t) : = 
\begin{pmatrix}
I_d & 0_d \\
t I_d & I_d
\end{pmatrix},
\ee
we have $\gamma (t) = A(t) D_z$. Thus, for any multi-index $\beta$ of order $|\beta| = s$,
\be
| \gamma^\beta(t) f | = |(A(t) D_z)^\beta f| \leq \| A(t) \|^s \sum_{|\beta '| = s} |\partial^{\beta '} f| \leq (1+t)^s \sum_{|\beta '| = s} |\partial^{\beta '} f| .
\ee
Hence
\be
\| \gamma^\beta(t) f \|_{L^2_z} \leq C_{s,d} (1+t)^s \| f \|_{\dot H^s_z} .
\ee
For the reverse inequality, we note that $D_z = A^{-1}(t) \gamma(t)$, and that $\| A^{-1}(t) \| \le (1+t)$. The estimate therefore follows by the same argument.

\end{proof}

We may therefore use the operators $\gamma$ to define (semi-)norms for the Sobolev spaces $H^s_z$ and $\dot H^s_z$ that are adapted to the transport flow. 

\begin{definition}
Let $t \in [0,T]$.
The semi-norm $\| \cdot \|_{\dot \Gamma^s(t)}$ is defined, for any function $f \in \dot H^s(\dom)$, by
\be
\| f \|_{\dot \Gamma^s(t)}^2 : = \sum_{|\beta| = s} \| \gamma^\beta(t) f \|_{L^2_z}^2 .
\ee
The norm $\| \cdot \|_{\Gamma^s(t)}$ is defined, for any function $f \in H^s(\dom)$,
by
\be
\| f \|_{\Gamma^s(t)}^2 : = \| f \|_{\dot \Gamma^s(t)}^2 + \| f \|_{L^2_z}^2 .
\ee
\end{definition}

Note that for each $t \in [0,T]$, by Lemma~\ref{lem:VF-conversion}, $\| \cdot \|_{\dot \Gamma^s(t)}$ and $\| \cdot \|_{ \Gamma^s(t)}$ are equivalent respectively to the standard (semi)-norms $\| \cdot \|_{\dot H^s}$ and $\| \cdot \|_{H^s}$. \\

In the next subsection, we will consider Sobolev estimates for compositions of functions. It will be useful for later sections to phrase these estimates in terms of the modified norms $\| \cdot \|_{\Gamma^s}$. 
For this purpose it will be convenient for us to make use of the transport \emph{flow map}.

\begin{definition} \label{def:FlowMap}
The \emph{flow map} $\phi : [0,T] \times \dom \to \dom$ is defined such that, writing $\phi = (\phi_X, \phi_V)$
\be
\partial_t \phi_X = \phi_V, \; \partial_t \phi_V = 0; \qquad \phi(0,z) = z, \; \forall z \in \dom .
\ee
Explicitly, 
\be
\phi(t,x,v) = (x+vt, v) .
\ee
\end{definition}

\begin{rmk}
The key property of $\phi$ that we will use is that, for any $f \in C^1(\dom)$,
\be
D_z (f \circ \phi) = (\gamma f) \circ \phi .
\ee
We note moreover that $\phi(t, \cdot) : \dom \to \dom$ is a diffeomorphism for all $t \in[0,T]$, with $\det D_z {\phi(t, \cdot)} \equiv 1$.

Hence
\be
\| f \|_{\dot \Gamma^s(t)}^2 = \sum_{|\beta| = s} \| [\partial^\beta (f \circ \phi(t, \cdot))] \circ \phi^{-1}(t, \cdot) \|_{L^2}^2 = \| f \circ \phi(t, \cdot) \|_{\dot H^s_z}^2 ,
\ee
and thus also
\be \label{eq:FlowSobolev}
\| f \|_{\Gamma^s(t)} = \| f \circ \phi(t, \cdot) \|_{H^s_z} .
\ee
\end{rmk}

\begin{rmk}
In the case of the general constant coefficient control system \eqref{eq:SDEGeneral}, one would consider the flow map satisfying
\be
\partial_t \phi = B \phi ; \qquad \phi(0,z) = z, \; \; \forall z \in \dom ,
\ee
i.e. $\phi(t,z) = e^{tB} z$.
\end{rmk}

\subsection{Sobolev Estimates for Composite Functions}

The goal of this subsection is to establish Proposition~\ref{prop:HRegularity} below,
showing that the operations of composition with the Hamiltonian $H$ and terminal coupling $g$
define continuous and bounded maps between suitable functional spaces. 

For convenience we introduce the following shorthand notation for the function 
$$(t,z,m,p)\mapsto m D_p H(t,z, m, p),$$ which we will use globally throughout the paper.

\begin{definition} \label{def:J}
The function 
$J:(0,T)\times\dom\times\R\times\R^d\to \R^d$ is defined by
\be
J(t, z,m,p) : = m D_p H(t, z, m, p).
\ee

\end{definition}

\begin{definition}[Composition operators] \label{def:Ops}
The (nonlinear) operators $\mc{H}$, $\mc{J}$ are defined for functions $m : [0,T] \times \dom \to \R$ and $p : [0,T] \times \dom \to \R^d$ by the formal expressions
\be
\mc{H} [m, p] : = H \circ (\iota_{[0,T] \times \dom}, m, p), \qquad \mc{J} [m, p] : = J \circ (\iota_{[0,T] \times \dom}, m, p).
\ee
The operator $\mc{G}$ is defined for functions $\mu : \dom \to \R$ by
\be
\mc{G} [ \mu ] : = g \circ (\iota_\dom, \mu) .
\ee
\end{definition}

We will show that under Assumption~\ref{hyp:terminal} the operator $\mc{G}$ is Lipschitz continuous and bounded on bounded subsets of $H^s_z$, and that under Assumption~\ref{hyp:main}, the operators $\mc{H}$, $\mc{J}$ are Lipschitz continuous and bounded on bounded subsets of the function space $L^\infty_t H^{s-1}_z \cap L^2_t H^s_z$. The choice of this function space is suggested by the regularization properties of the kinetic diffusion operator $\partial_t + v \cdot D_x - \frac{1}{2} \Delta_v$.
The key point is that the highest order derivatives in $z$ ($H^s_z$) require merely an $L^2_t$ bound with respect to time, while only derivatives of strictly lower order ($H^{s-1}_z$) require an $L^\infty_t$ control. This reflects the fact that we expect a gain of one derivative in $v$ from the diffusion, but in an $L^2_t$ sense (see Section~\ref{linearSection}).
It is working in this function space, therefore, that will allow us to compensate the loss of a $v$ derivative in $u$ inside the nonlinearities $H$ and $D_p H$ by the regularizing effect of the degenerate diffusion.

\begin{prop} \label{prop:HRegularity}
Let $s \geq s_\ast + 1$.
\begin{enumerate}[(i)]
\item Let $H$ satisfy Assumption~\ref{hyp:main}. 
The operators $\mc{H}$ and $\mc{J}$ given by the expressions in Definition~\ref{def:Ops} are well-defined as maps from $(L^\infty_t H^{s-1}_z \cap L^2_t H^s_z)^{d+1}$ into, respectively, $L^2_t H^s_z$ and $(L^2_t H^s_z)^d$.
Moreover these maps are bounded and Lipschitz continuous on bounded sets.

More precisely, there exists a non-decreasing function $F:[0,+\infty)\to[0,+\infty)$ such that, 
for any $(m,p) \in (L^2_t H^s_z \cap L^\infty_t H^{s-1}_z)^{d+1}$,
\be
\| \mc{H} [m, p] \|_{L^2_t \Gamma^s_z}^2 + \| \mc{J} [m, p] \|_{L^2_t \Gamma^s_z}^2 \leq \left(T + \| (m, p) \|_{L^2_t \Gamma^s_z}^2\right) \, (1+T)^{2s} \, F \left ( \| (m,p) \|_{L^\infty_t \Gamma^{s-1}_z}^2 \right ),
\ee
and for any $(m_i ,p_i) \in (L^2_t H^s_z \cap L^\infty_t H^{s-1}_z)^{d+1}$, $i=1,2$,
\begin{align}
& \| \mc{H} [m_1, p_1] -  \mc{H} [m_2, p_2]  \|_{L^2_t \Gamma^s_z}^2 + \| \mc{J} [m_1, p_1] -  \mc{J} [m_2, p_2]  \|_{L^2_t \Gamma^s_z}^2 
\leq 
 (1+T)^{2s} \, F \left ( \| (m,p) \|_{L^\infty_t \Gamma^{s-1}_z}^2 \right ) \\
& \qquad \times \left ( \| (m_1 - m_2, p_1 - p_2) \|_{L^2_t \Gamma^s_z}^2 + \| (m_1 - m_2, p_1 - p_2) \|_{L^\infty_t \Gamma^{s-1}_z}^2 (T + \max_i \| (m_i, p_i) \|_{L^2_t \Gamma^s_z}^2)  \right ) .
\end{align}

\item Let $g$ satisfy Assumption~\ref{hyp:terminal}.
The operator $\mc{G}$ given by the expression in Definition~\ref{def:Ops} is well-defined as a map from $H^s_z$ into $H^s_z$, and is bounded and Lipschitz continuous on bounded subsets of $H^s_z$. 

More precisely, there exists a non-decreasing function $G :[0,+\infty)\to[0,+\infty)$ such that, for any $\mu \in H^s_z$,
\be
\| \mc{G} [\mu] \|_{\Gamma^s(T)}^2 \leq (1+T)^{2s} \, G \left ( \| \mu \|_{\Gamma^s(T)}^2 \right ) ,
\ee
and for any $\mu_i \in H^s_z$, $i = 1,2$,
\be
\| \mc{G} [\mu_1] - \mc{G} [\mu_2] \|_{\Gamma^s(T)}^2 \leq (1+T)^{2s} \, G \left ( \| \mu \|_{\Gamma^s(T)}^2 \right )  \, \| \mu_1 - \mu_2 \|_{\Gamma^s(T)}^2 .
\ee

\end{enumerate}
\end{prop}

For the proof of this proposition, we will require certain Sobolev regularity estimates on composite functions of the form $z \mapsto \Theta(z, h(z))$, where we are given $h \in H^s_{z}(\dom;\R^n)$. We consider functions $\Theta$ satisfying the following hypotheses. 

\begin{hyp} \label{hyp:Theta}
The function $\Theta :  \dom \times \R^n \to \R^l$ satisfies, for some integer $s \geq s_\ast$: 
\begin{itemize}
\item $\dom  \times \R^n\ni(z,h)\mapsto \Theta(z,h)$ 
is of class $C^{s}$, with bounds uniform in $z$ and locally uniform in $h$. 
That is, for all $0 \leq s' \leq s$ there exists a non-decreasing function $\Phi_\Theta^{(s')}:[0,+\infty)\to[0,+\infty)$ such that, for all $r > 0$,
\be
\|\Theta\|_{C^{s'}(\dom \times B_r(0))} \leq \Phi_\Theta^{(s')}(r) .
\ee
\item Pure $z$-derivatives are $L^2$ when composed with an { $L^2 \cap L^\infty$} function, with locally uniform bounds, i.e. for all $0\leq s' \leq s$ there exists a non-decreasing function $\tilde \Phi_\Theta^{(s')}:[0,+\infty)\to[0,+\infty)$ such that
\be
\left \| \partial^\alpha_z \Theta \circ (\iota_{ \dom }, h) \right \|_{L^2} \leq \tilde \Phi_{\Theta}^{(s')} \left ( \| h \|_{L^2 \cap L^\infty} \right )  \qquad \forall h \in (L^2 \cap L^\infty)(\dom) , \, \alpha \in (\mb{N}\cup\{0\})^{m},  |\alpha| \leq s' .
\ee
\end{itemize}
Without loss of generality we will take $\Phi_\Theta^{(s')} = \tilde \Phi_\Theta^{(s')}$.
\end{hyp}
For $\Theta$ satisfying Assumption~\ref{hyp:Theta}, we define $s(\Theta)$ to be the maximal integer $s \geq s_\ast$ such that Assumption~\ref{hyp:Theta} holds for this value of $s$ (if Assumption~\ref{hyp:Theta} holds for all $s \geq s_\ast$, then $s(\Theta) : = + \infty$).

We then have the following Sobolev estimate on the composition $\Theta \circ (\iota_{\dom}, h)$ (see the discussion in \cite{ambroseMFG4}), which is phrased in terms of the standard $H^s$ norms.

\begin{lemma} \label{lem:composition}
Let $\Theta : \dom \times \R^n \to \R^l$ be a function satisfying Assumption~\ref{hyp:Theta}. Let $h: \dom \to \R^n$ be an $H^{s}$ function, where $s$ is an integer such that $s(\Theta) \geq s \geq s_\ast$. Let $\iota=\iota_{ \dom }$. Then the composition $\Theta \circ (\iota, h) : \dom \to \R^l$ is of class $H^{s}$ with 
\begin{align}
\| \Theta \circ (\iota, h) \|_{L^2_z} & \leq \Phi_\Theta^{(0)}(\| h \|_{L^2 \cap L^\infty}), \\
\| \Theta \circ (\iota, h) \|_{\dot H^{s}_z} & \leq  C_{s,d} \left( 1 + \| h \|_{H^{s}}^{s} \right )
 \Phi_\Theta^{(s)}(\| h \|_{L^2 \cap L^\infty}), 
\end{align}
where $\Phi_\Theta^{(0)},\Phi_\Theta^{(s)}:[0,+\infty)\to[0,+\infty)$ are the functions defined in Assumption~\ref{hyp:Theta}.

\end{lemma}

If $s \geq s_\ast + 1$, we may obtain the following refined dependence on the highest order derivative of $h$.
 
 \begin{lemma} \label{lem:composition-critplusone}
 Let $s$ be an integer such that $s \geq s_{\ast} + 1$.
 Assume that $\Theta : \dom \times \R^n \to \R^l $ satisfies Assumption~\ref{hyp:Theta} for $s(\Theta) \geq s$. Let $h: \dom  \to \R^n$ be a function of class $H^{s}$ and let $\iota = \iota_{ \dom }$. Then the composition $\Theta \circ (\iota, h)$ is an $H^{s}(\dom)$ function. Furthermore, there exists a non-decreasing function $\Psi_\Theta^{(s)} : \R_+ \to \R_+$ (depending on $\Theta$) such that the following estimate holds:
  \be
 \| \Theta \circ (\iota, h) \|_{\dot H^{s}} \leq (1+ \| h \|_{H^{s}}) \,\Psi_\Theta^{(s)} \left ( \| h \|_{H^{s-1}} \right ) .
 \ee
 \end{lemma}

The proof requires the classical Kato--Ponce inequality.

\begin{lemma} {\cite[Lemma X4]{Kato-Ponce}} \label{lem:Kato-Ponce}
Let $f,h \in H^s \cap L^\infty$. Then
\be
\| f h \|_{H^s} \leq C_s \left ( \| f \|_{H^s} \| h \|_{L^\infty} +  \| f \|_{L^\infty} \| h \|_{H^s} \right ) .
\ee

\end{lemma}

\begin{proof}[Proof of Lemma~\ref{lem:composition-critplusone}]

Let $\beta \in (\NN \cup \{0\} )^{m}$ be a multi-index of order $|\beta| = s$. Then there exists $j$ such that $\beta_j \geq 1$, and thus the multi-index $\beta - e_j$ has non-negative coordinates. By the chain rule for Sobolev functions we may therefore write
\be
\partial^\beta \left [ \Theta \circ (\iota, h) \right ] = \partial^{\beta - e_j} \partial_j(\Theta \circ (\iota, h)) = \partial^{\beta - e_j} \left ( \partial_{z_j} h \cdot ( D_h \Theta \circ (\iota, h)) + \partial_{z_j} \Theta \circ (\iota, h) \right ) .
\ee
Observe that $\partial_{z_j} \Theta$ satisfies Assumption~\ref{hyp:Theta} at order $s({D_z \Theta}) = s(\Theta) -1$. Hence, by Lemma~\ref{lem:composition}, $\partial_{z_j} \Theta \circ (\iota, h) \in H^{s-1}$, with the estimate
\be
\left\| \partial^{\beta - e_j} \left [ \partial_{z_j} \Theta \circ (\iota, h) \right ] \right\|_{L^2} \leq C_{s,d}  \left( 1 + \| h \|_{H^{s-1}}^{s-1} \right ) \Phi_{D_z \Theta}^{(s-1)}(\| h \|_{L^2 \cap L^\infty}) .
\ee

To control $\partial^{\beta - e_j} \left ( \partial_{z_j} h \cdot ( D_h \Theta \circ (\iota, h)) \right )$, we first work under the additional assumption that 
$D_h \Theta$ satisfies Assumption~\ref{hyp:Theta} with $s(D_h \Theta) \geq s-1$, i.e. we assume $L^2$ estimates on the following derivatives
\be \label{L2Dh}
\left \| \partial^\alpha_z D_h \Theta \circ (\iota_{ \dom }, h) \right \|_{L^2} \leq \Phi_{D_h \Theta}^{(s-1)} \left ( \| h \|_{L^2 \cap L^\infty} \right )  \quad \forall h \in L^2 \cap L^\infty(\dom) , \,  |\alpha| \leq s -1  .
\ee

Then, by the Kato--Ponce inequality (given in Lemma \ref{lem:Kato-Ponce}):
\be
\|  \partial_{z_j} h \cdot D_h \Theta \circ (\iota, h) \|_{\dot H^{s-1}} \leq C_{s,d} \left ( \| \partial_{z_j} h \|_{\dot H^{s-1}} \| D_h \Theta \circ (\iota, h) \|_{L^\infty} + \| \partial_{z_j} h \|_{L^\infty} \| D_h \Theta \circ (\iota, h) \|_{\dot H^{s-1}} \right)
\ee
By Lemma~\ref{lem:composition}, 
\be
\| D_h \Theta \circ (\iota, h) \|_{\dot H^{s-1}} \leq C_{s,d}  \left( 1 + \| h \|_{H^{s-1}}^{s-1} \right ) \Phi_{D_h \Theta}^{(s-1)}(\| h \|_{L^2 \cap L^\infty}) .
\ee

We then observe that in fact this estimate holds without the additional estimates \eqref{L2Dh}. The role of the $L^2$ estimates on pure $z$ derivatives in the proof of Lemma~\ref{lem:composition} is to control in $L^2$ the terms of the form $\partial^\alpha_z \Theta \circ (\iota, h)$ in the expansion of $\partial^\alpha_x \left ( \Theta \circ (\iota, h) \right )$. In the present setting,
scrutiny of the Leibniz/Faa di Bruno expressions for derivatives $\partial^\alpha \left ( \partial_{z_j} h \cdot D_h \Theta \circ (\iota, h) \right )$
shows that
quantities of the form $\partial^\alpha_z D_h \Theta \circ (\iota_{ \dom }, h)$
only appear in products with derivatives of $h$. A separate $L^2$ estimate is therefore unnecessary.

It follows that
\begin{multline}
\| \partial^\beta [ \Theta \circ (\iota, h) ] \|_{L^2} \leq  C_{s,d}( 1 + \| Dh \|_{L^\infty}) \left( 1 + \| h \|_{H^{s-1}}^{s-1} \right ) \Phi_{D_{z,h} \Theta}^{(s-1)}(\| h \|_{L^2 \cap L^\infty}) .  \\
+ C_{s,d} \Phi_{\Theta}^{(1)}(\| h \|_{L^2 \cap L^\infty}) \| h \|_{\dot H^{s}} .
\end{multline}

Then, since $\| Dh \|_{L^\infty} \leq C_d \| h \|_{H^{s_\ast + 1}}$ and $s \geq s_\ast + 1$, we conclude that
\begin{align*}
\| \partial^\beta(\Theta \circ ( \iota, h) \|_{L^2} & \leq 
C_{s,d}( 1 + \| h \|_{H^{s}}) \left( 1 + \| h \|_{H^{s-1}}^{s-1} \right ) \Phi_{ \Theta}^{(s)}(C_d \| h \|_{H^{s_\ast + 1}} ) , 
\end{align*}
which completes the proof.

\end{proof}

\begin{lemma} \label{lem:Hs-Lipschitz}
 Let $s$ be an integer such that $s \geq s_{\ast}+1$. Let $\Theta :  \dom \times \R^n \to \R^l $ be given.
Assume that $\Theta : \R^m \times \R^n \to \R^l $ satisfies Assumption~\ref{hyp:Theta} for $s(\Theta) \geq s + 1$. 
Let $h_1, h_2: \dom \to \R^n$ be $H^{s}$ functions. 
Then there exist non-decreasing functions $\Lambda_\Theta^{(0)}, \Lambda_\Theta^{(s)} : \R_+ \to \R_+$ such that
\begin{align}
\| \Theta \circ (\iota, h_1)  - \Theta \circ (\iota, h_2) \|_{L^2} & \leq \| h_1 - h_2 \|_{L^2} \, \Lambda_\Theta^{(0)} \left (  \max_i \| h_i \|_{H^{s-1}} \right ) \\
\| \Theta \circ (\iota, h_1)  - \Theta \circ (\iota, h_2) \|_{\dot H^s} & \leq \left ( \| h_1 - h_2 \|_{H^s} + (\| h_1 \|_{H^s} + \|h_2 \|_{H^s}) \| h_1 - h_2 \|_{H^{s_\ast}} \right ) \\& \qquad \times \Lambda_\Theta^{(s)} \left (  \max_i \| h_i \|_{H^{s-1}}\right) .
\end{align}
\end{lemma}
\begin{proof}
We first use the fundamental theorem of calculus 
 to write, for any $z \in \dom$, 
\be
\Theta(z, h_1(z)) - \Theta(z, h_2(z)) = \int_0^1 D_h \Theta(z, \eta h_1(z) + (1-\eta) h_2(z)) \cdot (h_1(z) - h_2(z)) \dd \eta .
\ee
Then, taking the $L^2$ norm with respect to $z$, by Minkowski's integral inequality we obtain
\be
\| \Theta \circ ( \iota, h_1) - \Theta \circ (\iota, h_2) \|_{L^2} \leq \int_0^1 \| D_h \Theta \circ (\iota, \eta h_1 + (1-\eta) h_2)\|_{L^\infty} \| h_1 - h_2 \|_{L^2} \dd \eta .
\ee
Then, since $\| \eta h_1 + (1-\eta) h_2 \|_{L^\infty} \leq \max_i \| h_i \|_{L^\infty}$ for all $\eta \in [0,1]$, it follows that
\begin{align}
\| \Theta \circ (\iota, h_1) - \Theta \circ (\iota, h_2) \|_{L^2} & \leq  \| \Theta \|_{C^1( \dom \times B_{\max_i \| h_i \|_{L^\infty}})} \| h_1 - h_2 \|_{L^2} \\
& \leq \Phi_\Theta^{(1)} \left ( \max_i \| h_i \|_{L^\infty} \right ) \,  \| h_1 - h_2 \|_{L^2} .
\end{align}

For higher order derivatives we have similarly the following bound
\be
\| \Theta \circ (\iota, h_1) - \Theta \circ (\iota, h_2) \|_{\dot H^s} \leq \int_0^1 \| D_h \Theta \circ (\iota, \eta h_1 + (1-\eta) h_2) \, \cdot \,( h_1 - h_2) \|_{\dot H^s} \dd \eta .
\ee
By the Kato--Ponce inequality (Lemma~\ref{lem:Kato-Ponce}),
\begin{multline}
\| \Theta \circ (\iota, h_1) - \Theta \circ (\iota, h_2) \|_{\dot H^s} \leq C_{s,d} \| h_1 - h_2 \|_{\dot H^s} \int_0^1 \| D_h \Theta \circ (\iota, \eta h_1 + (1-\eta) h_2) \|_{L^\infty}  \dd \eta \\
+ C_{s,d}  \| h_1 - h_2 \|_{L^\infty} \int_0^1 \| D_h \Theta \circ (\iota, \eta h_1 + (1-\eta) h_2) \|_{\dot H^s}  \dd \eta .
\end{multline}
Then, since $\| \eta h_1 + (1-\eta) h_2 \|_{L^\infty} \leq \max_i \| h_i \|_{L^\infty}$ for all $\eta \in [0,1]$, it follows that
\begin{align}
 \int_0^1 \| D_h \Theta \circ (\iota, \eta h_1 + (1-\eta) h_2) \|_{L^\infty} \dd \eta & \leq \| \Theta \|_{C^1(\dom \times B_{\max_i \| h_i \|_{L^\infty}})} \\
 & \leq \Phi_\Theta^{(1)} \left ( \max_i \| h_i \|_{L^\infty} \right ).
\end{align}
To estimate $\| D_h \Theta \circ (\iota, \eta h_1 + (1-\eta) h_2) \|_{\dot H^s}$, we apply Lemma~\ref{lem:composition-critplusone} to obtain
\begin{align}
\| D_h \Theta \circ (\iota, \eta h_1 + (1-\eta) h_2) \|_{\dot H^s} & \leq ( 1 +  \| \eta h_1 + (1-\eta) h_2 \|_{H^s}  ) \, \Psi_{D_h \Theta}^{(s)} \left ( \| \eta h_1 + (1-\eta) h_2 \|_{H^{s-1}} \right ) \\
& \leq (1 + \eta \| h_1 \|_{H^s} + (1-\eta) \| h_2 \|_{H^s}) \, \Psi_{D_h \Theta}^{(s)} \left ( \max_i \| h_i \|_{H^{s-1}} \right ) .
\end{align}
Hence
\be
\int_0^1 \| D_h \Theta \circ (\iota, \eta h_1 + (1-\eta) h_2) \|_{\dot H^s} \dd \eta \leq ( 1 + \frac{1}{2} \left ( \| h_1 \|_{H^s} + \| h_2 \|_{H^s} \right ) ) \, \Psi_{D_h \Theta}^{(s)} \left ( \max_i \| h_i \|_{H^{s-1}} \right ) .
\ee
As in the proof of Lemma~\ref{lem:composition-critplusone}, $L^2$ estimates on $\partial^\alpha_z D_h \Theta$ are not actually required for this estimate, since these quantities only appear in products with derivatives of $h_1 - h_2$ and $\eta h_1 + (1-\eta) h_2$.

Combining these estimates gives
\begin{align}
\| \Theta \circ (\iota, h_1) & - \Theta \circ (\iota, h_2) \|_{\dot H^s}  \leq C_{s,d} \Phi_\Theta^{(1)} \left ( \max_i \| h_i \|_{L^\infty} \right ) \| h_1 - h_2 \|_{\dot H^s} \\
& + C_{s,d} \left( 1 +   \| h_1 \|_{H^s} + \| h_2 \|_{H^s}  \right) \,  \|h_1 - h_2 \|_{L^\infty} \, \Psi_{D_h \Theta}^{(s)} \left ( \max_i \| h_i \|_{H^{s-1}} \right ) .
\end{align}
Finally, the estimates $ \|h_1 - h_2 \|_{L^\infty_z} \leq C_d  \|h_1 - h_2 \|_{H^{s_\ast}_z}$ and $ \|h_i \|_{L^\infty_z} \leq C_d  \|h_i \|_{H^{s_\ast}_z}$ complete the proof.

\end{proof}

Next, we will rephrase these composition estimates in terms of our adapted norms $\| \cdot \|_{\Gamma^s(t)}$.
In the present case we can obtain this straightforwardly from the estimates in standard norms by using the flow map $\phi$ introduced in Definition~\ref{def:FlowMap}.

\begin{lemma} \label{lem:CompositionFlow}
Let $s$ be an integer such that $s \geq s_{\ast} + 1$.
 Assume that $\Theta : \dom \times \R^n \to \R^l $ satisfies Assumption~\ref{hyp:Theta} for $s(\Theta) \geq s$.
 Then there exist non-decreasing functions $\Psi_\Theta^{(s)}, \Lambda_\Theta^{(s)} : \R_+ \to \R_+$ (depending on $\Theta$, $s$ and the dimension $d$) such that the following estimates holds:
\begin{enumerate}[(i)]
\item \label{item:uniform} Let $h: \dom  \to \R^n$ be a function of class $H^{s}$. Then
\be
\| \Theta \circ (\iota_\dom, h) \|_{\Gamma^s(t)} \leq (1 + \| h \|_{\Gamma^s(t)}) \,  (1+ t)^s \Psi_{\Theta}^{(s)} \left ( \| h \|_{\Gamma^{s-1}(t)} \right )  \quad \forall t \in [0,T] .
\ee
\item \label{item:difference} Let $h_1, h_2: \dom  \to \R^n$ be $H^{s}$ functions. Then
\begin{align}
\| \Theta \circ (\iota, h_1)  - \Theta \circ (\iota, h_2) \|_{ \Gamma^s(t)} & \leq \left ( \| h_1 - h_2 \|_{\Gamma^s(t)} + (\| h_1 \|_{\Gamma^s(t)} + \|h_2 \|_{\Gamma^s(t)}) \| h_1 - h_2 \|_{\Gamma^{s_\ast}(t)} \right ) \\
& \qquad \qquad \times (1+t)^s \Lambda_\Theta^{(s)} \left (  \max_i \| h_i \|_{\Gamma^{s-1}(t)}\right)
\end{align}
\end{enumerate}
\end{lemma}

\begin{rmk}
If the function $\Theta$ does not depend explicitly on the $z$ variable, then the additional factors of $(1+t)^s$ do not appear.
\end{rmk}

\begin{proof}
Using the identity \eqref{eq:FlowSobolev}, we have
\be
\| \Theta \circ (\iota_\dom, h) \|_{\dot \Gamma^s(t)} = \| \Theta \circ (\phi, h \circ \phi ) \|_{\dot H^s_z} .
\ee

We will check that  the function $\tilde \Theta : = \Theta \circ (\phi(t, \cdot), \iota_{\R^n})$ 
satisfies Assumption~\ref{hyp:Theta}. Indeed, for any multi-index $\alpha$ of order $|\alpha| = s'$,
\be
\left | \partial^\alpha_z \tilde \Theta (z,r) \right | = \left | (A(t)D_z)^\alpha \Theta (z,r) \right | \leq C_\alpha (1+t)^{s'}  \sum_{|\beta|=s'} \left | \partial_{z}^{\beta} \Theta (z,r)  \right |, 
\ee

Hence
\be
\| \tilde \Theta \|_{C^s(\dom \times B_r(0))} \leq C_\alpha (1+t)^{s} \| \Theta \|_{C^s(\dom \times B_r(0))} \leq  C_\alpha (1+t)^{s} \Phi_{\Theta}^{(s)}(r) 
\ee
and, for any function $h \in L^2 \cap L^\infty$ and multi-index $\alpha$ of order $|\alpha| = s'$,
\begin{align}
\left \| \partial^\alpha_z \tilde \Theta \circ (\iota_{\dom }, h) \right \|_{L^2_z} &\leq C_\alpha (1+t)^{s'} \sum_{|\beta|=s'} \left \| \partial_{z}^{\beta} \Theta \circ (\iota_\dom, h)  \right \|_{L^2_z} \\
& \leq C_\alpha (1+t)^{s'} \Phi_\Theta^{(s')}( \| h \|_{L^2_z \cap L^\infty_z} ).
\end{align}
Thus $\tilde \Theta$ satisfies Assumption~\ref{hyp:Theta} with $s(\tilde \Theta) = s(\Theta)$ and
\be
\Phi_{\tilde \Theta}^{(s')} \lesssim (1+t)^{s'} \Phi_\Theta^{(s')} \quad \forall s ' \leq s(\Theta) .
\ee

Then, by Lemma~\ref{lem:composition-critplusone},
\begin{align}
\| \Theta \circ (\iota_\dom, h) \|_{\dot \Gamma^s(t)} & \leq (1+ \| h \circ \phi(t, \cdot)  \|_{H^{s}_z}) \,\Psi_{\tilde \Theta}^{(s)} \left ( \| h \circ \phi(t, \cdot) \|_{H^{s-1}_z} \right ) \\
& \leq C_{s,d} (1+t)^{s} (1+ \| h \|_{\Gamma^{s}(t)}) \,\Psi_{\Theta}^{(s)} \left ( \| h  \|_{\Gamma^{s-1}(t)} \right ) .
\end{align}
We redefine the function $\Psi_{\Theta}^{(s)}$ to absorb the uniform constant $C_{s,d}$, and the statement \eqref{item:uniform} follows. Statement \eqref{item:difference} can be proved by a similar argument using Lemma~\ref{lem:Hs-Lipschitz}.

\end{proof}

As the final part of this section, we may now complete the proof of Proposition~\ref{prop:HRegularity}.

\begin{proof}[Proof of Proposition~\ref{prop:HRegularity}]
It follows directly from Assumption~\ref{hyp:main} that $H$ satisfies the assumptions of Lemma~\ref{lem:CompositionFlow}. 
We will show that the same is true for the function $J$ (recall Definition~\ref{def:J}).

It is clear that $J$ is continuous in all arguments and $C^{s+1}$ in $(z,m,p)$. For the $L^2$ estimate, if $|\alpha| \leq 1$ with $\alpha_m = 0$ and $0 \leq \beta \leq s$, then 
$ \partial^\alpha_{z,p} \partial^\beta_z J = m \partial^\alpha_{z,p} \partial^\beta_z D_p H$. Then, for each $t \in [0,T]$ and any function $m \in (L^2_z \cap L^\infty_z)(\dom)$,
\begin{align}
\|  \partial^\alpha_{z,p} \partial^\beta_z J \circ (t, \iota_{\dom}, m, p)\|_{L^2_z} & = \| m \, \partial^\alpha_{z,p} \partial^\beta_z D_p H  \circ (t, \iota_{\dom}, m, p) \|_{L^2_z} \\
& \leq \| m \|_{L^2_z} \| \partial^\alpha_{z,p} \partial^\beta_z D_p H \circ (t, \iota_{\dom}, m, p) \|_{L^\infty} \\
& \leq \| H(t, \cdot, \cdot, \cdot) \|_{C^{s+2}_{(z,m,p)}(\dom \times B_{\|(m,p)\|_{L^\infty_z}})} \| m \|_{L^2_z} \\
& \leq \Phi_H(\|(m,p)\|_{L^\infty_z} ) \| m \|_{L^2_z}  .
\end{align}
Otherwise, if $\partial^\alpha = \partial_m$, then $ \partial_m \partial^\beta_z J = \partial^\beta_z D_p H + m \partial_m \partial^\beta_z D_p H$. Hence
\be
\left \| \partial_m  \partial^\beta_z J  \circ (t, \iota_{\dom}, m, p) \right \|_{L^2_z} \leq \| \partial^\beta_z D_p H \circ (t, \iota_{\dom}, m, p) \|_{L^2_z} + \|  m \,  \partial_m \partial^\beta_z D_{p} H \circ (t, \iota_{\dom}, m, p) \|_{L^2_z} . 
\ee
As before, we have
\begin{align}
\|  m \,  \partial_m \partial^\beta_z D_{p} H \circ (t, \iota_{\dom}, m, p) \|_{L^2_z} & \leq \| H(t, \cdot, \cdot, \cdot) \|_{C^{s+2}_{(z,m,p)}(\dom \times B_{\|(m,p)\|_{L^\infty_z}}) } \| m \|_{L^2_z} \\
& \leq \Phi_H(\|(m,p)\|_{L^\infty_z} ) \| m \|_{L^2_z}  .
\end{align}
By the $L^2$ estimates for $D_p H$ in Assumption~\ref{hyp:main} we have
\be
\| \partial^\beta_z D_p H \circ (t, \iota_{\dom}, m, p) \|_{L^2_z} \leq \Phi_H(\|(m,p)\|_{L^\infty_z} ).
\ee
Hence
\be
\left \| \partial_m  \partial^\beta_z J  \circ (t, \iota_{\dom}, m, p) \right \|_{L^2_z} \leq (1 + \|  m \|_{L^2_z} ) \, \Phi_H \left ( \| (m, p) \|_{L^2_z \cap L^\infty_z} \right ) .
\ee
Thus $J$ also satisfies the assumptions of Lemma~\ref{lem:CompositionFlow}. Applying this lemma to $H$ and $J$ and taking the $L^2_t$ norm completes the proof.

\end{proof}

\section{Linear estimates}\label{linearSection}

The goal of this section is to establish energy estimates for solutions of the initial value problem for the linear Kolmogorov equation with source:
\be 
\left(\partial_t + v \cdot D_x - \frac{1}{2} \Delta_v\right) w = h, \qquad w \vert_{t=0} = w_0 ,
\ee
and to apply them in the setting of a decoupled MFG system that anticipates the fixed point argument we will use to construct solutions.
In Subsection \ref{linearEquation} below we develop estimates for solutions of a single linear equation.
Then, in Subsection \ref{sec:MFGEstimates}, we consider the decoupled MFG-type equations, obtaining bounds on both individual solutions and differences between solutions with different source terms.

\subsection{Estimates for solutions of a Kolmogorov equation}\label{linearEquation}

In the following lemma we recall an elementary estimate for solutions of the Kolmogorov equation that will be the key tool in our construction of solutions to the Mean Field Games system \eqref{eq:MFG}.
We first define the space of solutions we will consider. For any $s \geq 0$, let $X^s$ denote the set
\be
X^s : = \left \{ w \in C^0 \left ([0,T] ; H^s_z(\dom) \right ) : D_v w \in L^2 \left ((0,T) ; H^s_z(\dom) \right ) \right \} .
\ee

$X^s$ forms a Banach space when equipped with the norm $\| \cdot \|_{X^s}$, defined by
\be
\| w \|_{X^s}^2 : = \sup_{t \in [0,T]}\| w(t, \cdot) \|^2_{\Gamma_z^s(t)} + \int_0^T  \| D_v w(\tau, \cdot) \|_{\Gamma_z^s(\tau)}^2 \dd \tau .
\ee
We note that $\| \cdot \|_{X^s}$ is equivalent to the standard norm
$
\left ( \| w \|_{C^0_t H^s_z}^2 + \| D_v w \|_{L^2_t H^s_z}^2 \right )^{1/2}.
$

\begin{lemma} \label{lem:apriori}

Let $w \in X^0$ satisfy
\be \label{eq:Kol}
\left(\partial_t + v \cdot D_x - \frac{1}{2} \Delta_v\right) w = h_1 + \div_v h_2 
\ee
in the sense of distributions, where $h_1, h_2 \in L^2([0,T] \times \dom)$.
Then
\be
\sup_{t \in [0,T]}\| w(t) \|_{L^2_z}^2 + \int_0^T \| D_v w(t) \|_{L^2_{z}}^2 \dd t  \leq 2  \| w(0) \|_{L^2_z}^2 
 + 4 \int_0^T \left[T \|  h_1 (t) \|_{L^2_{z}}^2+  \| h_2 (t) \|_{L^2_{z}}^2\right] \dd t .
\ee
\end{lemma}
\begin{proof}
Formally test the equation \eqref{eq:Kol} with the function $w$; after integrating by parts in the diffusive term and the term involving $h_2$, we obtain
\be
\frac{1}{2} \| w(t) \|_{L^2_{z}}^2 + \frac{1}{2} \int_0^t \| D_v w(\tau) \|_{L^2_{z}}^2 \dd \tau  = \frac{1}{2} \| w(0) \|_{L^2_{z}}^2 + \int_0^t \int_{\dom}  h_1   w \dd z - \int_0^t \int_{\dom}  h_2 \cdot D_v w \dd z\dd \tau .
\ee
This can be made rigorous by an approximation argument, since $w \in C^0_t L^2_{x,v} \cap L^2_{t,x}\dot H^1_v$ (see Appendix~\ref{appendixSection} for the details).

By H\"older's inequality,
\be
 \| w(t) \|_{L^2_{z}}^2 + \int_0^t \| D_v w(\tau) \|_{L^2_{z}}^2 \dd \tau  \leq \| w(0) \|_{L^2_{z}}^2
+ 2 \int_0^t \| h_1(\tau) \|_{L^2_z} \| w (\tau) \|_{L^2_z} \dd \tau 
+ 2 \| h_2 \|_{L^2_{t,z}} \| D_v w \|_{L^2_{t,z}} .
\ee
By Young's inequality,
\begin{multline}
 \| w(t) \|_{L^2_z}^2 + \int_0^t \| D_v w(\tau) \|_{L^2_z}^2 \dd \tau  \leq \| w(0) \|_{L^2_z}^2 
+ \frac{1}{2 t}  \int_0^t \| w(\tau) \|_{L^2_z}^2 \dd \tau + 2 t\int_0^t \| h_1 (\tau) \|_{L^2_z}^2 \dd \tau \\
+ \frac{1}{2}  \int_0^t \| D_v w(\tau) \|_{L^2_z}^2 \dd \tau + 2 \int_0^t \| h_2 (\tau) \|_{L^2_z}^2 \dd \tau .
\end{multline}
Rearranging terms gives
\begin{multline}
\| w(t) \|_{L^2_z}^2 + \frac{1}{2} \int_0^t \| D_v w(\tau) \|_{L^2_z}^2 \dd \tau  \leq  \frac{1}{2}  \sup_{\tau \in [0,t]} \| w(\tau) \|_{L^2_z}^2 
+ \| w(0) \|_{L^2_z}^2  \\
+ 2 \int_0^t t \| h_1 (\tau) \|_{L^2_z}^2 +  \| h_2 (\tau) \|_{L^2_z}^2 \dd \tau. 
\end{multline}
We next take supremum over $t \in [0,T]$ and multiply by two to obtain
\begin{multline}
2 \sup_{t \in [0,T]}\| w(t) \|_{L^2_z}^2 + \int_0^T \| D_v w(\tau) \|_{L^2_z}^2 \dd \tau  \leq \sup_{t \in [0,T]} \| w(t) \|_{L^2_z}^2 \\ 
+ 2 \| w(0) \|_{L^2_z}^2 
 + 4 \int_0^T T \| h_1 (\tau) \|_{L^2_z}^2 \dd \tau + \| h_2 (\tau) \|_{L^2_z}^2 \dd \tau .
 \end{multline}
We rearrange once more to absorb the supremum term on the right hand side, which completes the proof:
\be
\sup_{t \in [0,T]}\| w(t) \|_{L^2_z}^2 +  \int_0^T \| D_v w(t) \|_{L^2_z}^2 \dd t  \leq  2 \| w(0) \|_{L^2_z}^2 
 + { 4} \int_0^T T \| h_1 (t) \|_{L^2_z}^2 + \| h_2 (t) \|_{L^2_z}^2 \dd t .
\ee
\end{proof}

\begin{cor} \label{cor:apriori-higher}
Let $w \in X^s$ ($s \in \mb{N}$) satisfy \eqref{eq:Kol}
in the sense of distributions, where $h_1, h_2 \in L^2_{t}H^s_{z}$.
Then
\be
\| w \|_{X^s}^2 \leq 2 \| w(0, \cdot) \|_{H^s_z}^2 + 4 \left ( T \| h_1 \|_{L^2_t \Gamma^s_z}^2 +  \| h_2 \|_{L^2_t \Gamma^s_z}^2 \right ) .
\ee

\end{cor}
\begin{proof}
Let $\beta \in (\NN \cup \{0\})^{2d}$ be a multi-index with $|\beta| = s$, and consider the function $\gamma^\beta w \in X^0$.
By the commutation relations noted in Remark~\ref{rmk:commute}, $\gamma^\beta w$ satisfies
\be
\left(\partial_t + v \cdot D_x - \frac{1}{2} \Delta_v\right) \gamma^\beta w = \gamma^\beta h_1 + \div_v \gamma^\beta h_2 .
\ee
Since $h_1, h_2 \in L^2_{t}H^s_{x,v}$ by assumption, we have $\gamma^\beta h_1, \gamma^\beta h_2 \in L^2_{t,z}$. 
Hence, by
Lemma~\ref{lem:apriori},
\begin{multline}
\sup_{t \in [0,T]}\| \gamma^\beta(t) w(t, \cdot ) \|_{L^2_z}^2 + \int_0^T \| \gamma^\beta(t) D_v w(t, \cdot) \|_{L^2_z}^2 \dd t \\
 \leq 2 \| \partial^{\beta} w(0, \cdot) \|_{L^2_z}^2 
 + 4 \int_0^T T \| \gamma^\beta(t) h_1 (t, \cdot) \|_{L^2_z}^2 + \| \gamma^\beta(t) h_2 (t, \cdot) \|_{L^2_z}^2 \dd t .
\end{multline}
Summing over all $\beta$ of order $|\beta| = s$, we obtain the estimate
\begin{multline}
\sup_{t \in [0,T]}\| w(t, \cdot ) \|_{\dot \Gamma^s_z(t)}^2 + \int_0^T \| D_v w(t, \cdot) \|_{\dot \Gamma^s_z(t)}^2 \dd t \\
\leq 2 \| w(0, \cdot) \|_{\dot H^s_z}^2 + 4 \int_0^T \left ( T \| h_1 (t, \cdot) \|_{\dot \Gamma^s_z(t)}^2 +  \| h_2 (t, \cdot) \|_{\dot \Gamma^s_z(t)}^2 \right ) \dd t .
\end{multline}
Combining this with the $L^2$ estimate from Lemma~\ref{lem:apriori} completes the proof.

\end{proof}

We collect the results of this subsection into the following proposition.

\begin{prop} \label{prop:Kol-WP}
Let $s \in \NN \cup \{0\}$, $w_0 \in H^s(\dom)$ and $h_1, h_2 \in L^2\left ((0,T) ; H^s(\dom) \right )$. Then there is a unique weak solution $w \in X^s$ of the problem
\be \label{eq:Kol-Hs}
\left(\partial_t + v \cdot D_x - \frac{1}{2} \Delta_v\right) w = h_1 + \div_v h_2, \qquad w \rvert_{t=0} = w_0 .
\ee
Furthermore,
$w$ satisfies the estimate
\be \label{est:Kol-Hs}
\| w \|_{X^s}^2
\leq  2 \| w_0 \|_{H^s_{z}}^2 + 4 \int_0^T \left[T\|  h_1(t, \cdot) \|_{\Gamma^s_z(t)}^2 +  \| h_2(t, \cdot) \|_{\Gamma^s_z(t)}^2\right] \dd t . 
\ee

\end{prop}
\begin{proof}
The proof of existence of solutions for this linear problem 
is classical. We use the representation formula in terms of the fundamental solution computed by Kolmogorov \cite{Kolmogorov} for equation \eqref{eq:Kol-Hs}; see \cite{Hormander} for the case of general drift-diffusion operators:
\be
S(t, z ' ) : = \left ( \frac{3}{\pi^2} \right )^{d/2} t^{-2d} e^{- \frac{1}{2t} \left ( 4 |v'|^2 + \frac{12}{t} v' \cdot x' + \frac{3}{t^2} |x'|^2 \right ) } .
\ee

If $w_0, h_1, h_2$ are smooth ($C^\infty$) functions, then the following formula defines a smooth solution of \eqref{eq:Kol}:
\be
w(t,z) = S(t, \cdot) \ast_z w_0 ( x - tv , v) 
+ \int_0^t S(\tau , \cdot) \ast_z [ h_1(t-\tau , \cdot) + \div_v h_2(t-\tau, \cdot)] (x - \tau v  , v ) \dd \tau  .
\ee
For general $w_0 \in H^s(\dom)$ and $h_1, h_2 \in L^2\left ([0,T] ; H^s(\dom) \right )$, take smooth approximating sequences $w_0^{(n)}, h_1^{(n)}, h_2^{(n)}$ such that
\be \label{Kol-data-approx}
\lim_{n \to +\infty} \left ( \| w_0^{(n)} - w_0 \|_{H^s} + \| h_1^{(n)} - h_1 \|_{L^2_tH^s_z} +  \| h_2^{(n)} - h_2 \|_{L^2_tH^s_z} \right ) = 0 .
\ee
Let $w^{(n)}$ be defined by the formula
\be
w^{(n)}(t,z) = S(t, \cdot) \ast_z w_0^{(n)} ( x - tv , v) 
+ \int_0^t S(\tau , \cdot) \ast_z [ h_1^{(n)}(t- \tau , \cdot) + \div_v h_2^{(n)}(t- \tau , \cdot)] (x - \tau v  , v ) \dd \tau  .
\ee
Then $w^{(n)}$ is an $X^s$ solution of
\be \label{eq:Kol-n}
\left(\partial_t + v \cdot D_x - \frac12 \Delta_v\right) w^{(n)} = h_1^{(n)} + \div_v h_2^{(n)}, \qquad w^{(n)} \rvert_{t=0} = w_0^{(n)} .
\ee
We will show that $w^{(n)}$ is a Cauchy sequence in $X^s$, and therefore converges.

Therefore let $n, j \in \NN$. Since the equation \eqref{eq:Kol} is linear in $w$, the difference $w^{(n)} - w^{(j)}$ satisfies
\be \label{eq:Kol-n-j}
\left(\partial_t + v \cdot D_x - \frac12 \Delta_v\right) (w^{(n)} - w^{(j)} )= ( h_1^{(n)} - h_1^{(j)} ) + \div_v ( h_2^{(n)} - h_2^{(j)} ), 
\ee
{with initial condition $( w^{(n)} - w^{(j)} ) \, \rvert_{t=0} = w_0^{(n)} - w_0^{(j)}.$}
By the estimates of Corollary~\ref{cor:apriori-higher},
\be
\| w^{(n)} - w^{(j)} \|_{X^s}^2 \leq 2 \| w_0^{(n)} - w_0^{(j)} \|_{H^s_z}^2 + 4 \left ( T \| h_1^{(n)} - h_1^{(j)} \|_{L^2_t \Gamma^s_z}^2 + \| h_2^{(n)} - h_2^{(j)} \|_{L^2_t \Gamma^s_z}^2 \right ).
\ee
By \eqref{Kol-data-approx}, the right hand side tends to zero as $n, j \to +\infty$. Hence the sequence $(w^{(n)})_{n=1}^{\infty}$ is indeed Cauchy in $X^s$.
In particular $(w^{(n)})_{n=1}^{\infty}$ is Cauchy in $C^0_tH^s_z$, and $(D_v w^{(n)})_{n=1}^{\infty}$ is Cauchy in $L^2_tH^s_z$,
and there exists a limit $w \in C^0_tH^s_z$ with distributional derivative $D_v w \in L^2_tH^s_z$, such that
\be \label{Kol-sol-approx}
\lim_{n \to +\infty} \left ( \| w^{(n)} - w \|_{L^\infty_t H^s_z}^2 + \| D_v (w^{(n)} - w) \|_{L^2_t H^s_z}^2 \right ) = 0 .
\ee

Finally, using \eqref{Kol-data-approx} and \eqref{Kol-sol-approx} we may pass to the limit in the weak form of the equation \eqref{eq:Kol-n} to deduce that $w$ is a $X^s$ distributional solution of \eqref{eq:Kol-Hs}.
The estimate \eqref{est:Kol-Hs} is then a direct consequence of Corollary \ref{cor:apriori-higher}.

\end{proof}

We conclude this subsection with some comments on the backward problem
\be \label{eq:Kol-bwd}
- \left(\partial_t + v \cdot D_x + \frac{1}{2} \Delta_v\right) w = h_1 + \div_v h_2, \qquad w \rvert_{t=T} = w_T .
\ee
This terminal value problem may be transformed into an initial value problem by a time reversal, i.e. by considering $\tilde w = w \circ (T - \iota_{[0,T]}, \iota_{\dom})$, which satisfies
\be \label{eq:Kol-flip}
\left(\partial_t - v \cdot D_x -  \frac{1}{2} \Delta_v\right) \tilde w = \tilde h_1 + \div_v \tilde h_2, \qquad \tilde w \rvert_{t=0} = w_T .
\ee
This is not of the form \eqref{eq:Kol} due to the appearance of the backward transport operator $\partial_t - v \cdot D_x$. In particular the vector fields $\gamma$ do not commute with $\partial_t - v \cdot D_x$, and so we cannot apply Proposition~\ref{prop:Kol-WP} immediately as we would in the case of the heat operator $\partial_t - \Delta_z$. However we may obtain estimates on $w$ in terms of the norm $\| \cdot \|_{X^s}$ by the following procedure. 

\begin{cor} \label{cor:Kol-bwd}
Let $w \in X^s$ be a distributional solution of equation \eqref{eq:Kol-bwd}, where $h_1, h_2 \in L^2_{t}H^s_{z}$. Then
\be
\| w \|_{X^s}^2 \leq  2 \| w(T, \cdot) \|_{\Gamma^s_z(T)}^2 
 +   4 \int_0^T \left[T \|  h_1 (t) \|_{\Gamma^s_z(t)}^2+  \| h_2 (t) \|_{\Gamma^s_z(t)}^2\right] \dd t .
 \ee

\end{cor}

\begin{rmk}
We emphasise that the norm appearing on the terminal datum is the $\Gamma^s_z(T)$ norm adapted to the terminal time $T$. 
This is particularly significant since we intend to use these estimates in the context of a coupled forward-backward problem.
The information we require on the terminal datum for the backward problem is of the same type as is output by the forward problem for its solution at time $T$. 
Thus, by working in the adapted norms throughout, we are able to avoid losing any additional factors of $(1+T)$ through the terminal coupling, that might otherwise arise from the conversion between $\Gamma^s$ and standard $H^s$ norms.
\end{rmk}

\begin{proof}[Proof of Corollary \ref{cor:Kol-bwd}]

The same argument as in Lemma~\ref{lem:apriori}, up to a change of sign, shows that the zero-order estimate holds also for equation \eqref{eq:Kol-flip}, hence
\be
\sup_{t \in [0,T]}\| w(T-t) \|_{L^2_z}^2 + \int_0^T \| D_v w(T-t) \|_{L^2_{z}}^2 \dd t  \leq  2 \| w_T \|_{L^2_z}^2 
 +   4 \int_0^T T \|  h_1 (T-t) \|_{L^2_{z}}^2+  \| h_2 (T-t) \|_{L^2_{z}}^2 \dd t .
\ee

The transformation $t \mapsto T-t$ results in the $X^0$ estimate for $w$:
\be \label{bwdX0}
\sup_{t \in [0,T]}\| w(t) \|_{L^2_z}^2 + \int_0^T \| D_v w(t) \|_{L^2_{z}}^2 \dd t  \leq  2 \| w_T \|_{L^2_z}^2 
 +   4 \int_0^T T \|  h_1 (t) \|_{L^2_{z}}^2+  \| h_2 (t) \|_{L^2_{z}}^2 \dd t .
\ee

Then, for the order $s$ estimate, we observe that $\gamma^\beta w$ solves the backward problem \eqref{eq:Kol-bwd} for any multi-index $\beta$ of order $|\beta| = s$. Hence, by the $X^0$ estimate \eqref{bwdX0} (which holds for any solution of the backward problem \eqref{eq:Kol-bwd}),
\be \label{bwdGs}
\sup_{t \in [0,T]}\| \gamma^\beta w(t) \|_{L^2_z}^2 + \int_0^T \| D_v \gamma^\beta w(t) \|_{L^2_{z}}^2 \dd t  \leq  2 \| \gamma^\beta(T) w_T \|_{L^2_z}^2 
 +   4 \int_0^T T \|  \gamma^\beta h_1 (t) \|_{L^2_{z}}^2+  \| \gamma^\beta h_2 (t) \|_{L^2_{z}}^2 \dd t ,
\ee
and we conclude by summing over $\beta$ that
\be \label{bwdXs}
\sup_{t \in [0,T]}\| w(t) \|_{\Gamma^s_z(t)}^2 + \int_0^T \| D_v w(t) \|_{\Gamma^s_z(t)}^2 \dd t  \leq  2 \| w_T \|_{\Gamma^s_z(T)}^2 
 +   4 \int_0^T T \|  h_1 (t) \|_{\Gamma^s_z(t)}^2+  \| h_2 (t) \|_{\Gamma^s_z(t)}^2 \dd t .
\ee

\end{proof}

\subsection{Estimates for the MFG equations} \label{sec:MFGEstimates}

We apply the estimates of Proposition~\ref{prop:Kol-WP} in the setting of MFGs systems. We consider the following slight generalisation of \eqref{eq:MFG}: 
\be \label{eq:MFG-gen}
\left\{
\begin{array}{l}
- \partial_t u - v \cdot D_x u - \frac{1}{2} \Delta_v u - \e \mc{H} [m, D_v u ] = 0,\\ [5pt] 
\partial_t m  + v \cdot D_x m - \frac{1}{2} \Delta_v m + \e \div_v \left ( \mc{J} [m, D_v u ] \right ) = 0,\\ [5pt] 
u \vert_{t = T} = \delta \, \mc{G} [ m(T, \cdot) ],\\ [5pt] 
m(0,z) = m^0, 
\end{array}
\right.
\ee
where now $\mc{H}, \mc{J}, \mc{G}$ denote continuous functions between the following spaces:
\begin{align}
&\mc{H} : \left ( L^2_t H^s_z \cap L^\infty_t H^{s-1}_z \right )^{d+1} \to L^2_t H^s_z \\
&\mc{J} : \left ( L^2_t H^s_z \cap L^\infty_t H^{s-1}_z \right )^{d+1} \to \left ( L^2_t H^s_z \right )^d \\
&\mc{G} : H^s_z \to H^s_z ,
\end{align}
that are bounded and Lipschitz on bounded subsets of their respective domains.
In quantitative terms, we assume the following.

\begin{hyp} \label{hyp:data-gen}
\begin{enumerate}[(i)]
\item \label{hyp:H-gen} There exists a function $\mc{F} : [0, + \infty)^2 \to [0, + \infty)$ which is non-decreasing with respect to both variables and such that, for functions $m \in L^2_t H^s_z \cap L^\infty_t H^{s-1}_z $ and $p \in \left ( L^2_t H^s_z \cap L^\infty_t H^{s-1}_z \right )^{d}$, we have
\be
\| \mc{H}[m,p] \|_{L^2_t \Gamma^s_z}^2 + \| \mc{J}[m,p] \|_{L^2_t \Gamma^s_z}^2 \leq \mc{F} \left (T, \| (m,p) \|_{L^2_t \Gamma^s_z \cap L^\infty_t \Gamma^{s-1}_z }^2 \right ) ;
\ee
and for $m_i \in L^2_t H^s_z \cap L^\infty_t H^{s-1}_z $ and $p_i \in \left ( L^2_t H^s_z \cap L^\infty_t H^{s-1}_z \right )^{d}$ ($i=1,2$),
\begin{multline}
\| \mc{H}[m_1,p_1] - \mc{H}[m_2,p_2]  \|_{L^2_t \Gamma^s_z}^2 + \| \mc{J}[m_1,p_1] - \mc{J}[m_2,p_2]  \|_{L^2_t \Gamma^s_z}^2 \\
 \leq \| (m_1 - m_2, p_1 - p_2 ) \|_{L^2_t \Gamma^s_z \cap L^\infty_t \Gamma^{s-1}_z}^2 \, \mc{F} \left (T, \max_i \| (m_i,p_i) \|_{L^2_t \Gamma^s_z \cap L^\infty_t \Gamma^{s-1}_z }^2 \right )  .
\end{multline}
\item \label{hyp:G-gen} There exists a function $\mc{K} : [0, + \infty)^2 \to [0, + \infty)$ which is non-decreasing with respect to both variables and such that, for $\mu \in H^s_z$,
\be \label{hyp:g-bdd-functional}
\| \mc{G} [\mu]  \|_{\Gamma^s_z(T)}^2 \leq \mc{K} \left ( T,  \| \mu \|_{\Gamma^s_z(T)}^2 \right ) ;
\ee
and for $\mu_i \in H^s_z$ ($i=1,2$),
\be \label{hyp:g-stability-functional}
\| \mc{G} [\mu_1] - \mc{G} [ \mu_2] \|_{\Gamma^s_z(T)}^2 \leq \| \mu_1 - \mu_2 \|_{\Gamma^s_z (T)}^2 \: \mc{K} \left ( T, \max_i \| \mu_i\|_{\Gamma^s_z(T)}^2 \right ).
\ee
\end{enumerate}

\end{hyp}

By Proposition~\ref{prop:HRegularity}, Assumption~\ref{hyp:data-gen} is satisfied by Hamiltonians and terminal couplings depending locally on $m$, for $H$ and $g$ satisfying Assumptions~\ref{hyp:main} and \ref{hyp:terminal}.
However, it can also apply to cases with non-local dependence on $m$, or mixed local and non-local dependence.

In anticipation of the fixed point argument we will use to construct solutions, we consider $u \in X^s$ satisfying 
\be \label{eq:value-gen}
\begin{cases}
- \partial_t u - v \cdot D_x u  - \frac{1}{2} \Delta_v u  =  \e \mc{H}[\tilde m, D_v \tilde u ]  \\
u \rvert_{t=T} = \delta \mc{G} [  \tilde m (T, \cdot) ]
\end{cases}
\ee
and $m \in X^s$ satisfying
\be \label{eq:cty-gen}
\begin{cases}
\partial_t m  + v \cdot D_x m  - \frac{1}{2} \Delta_v m  = - \e \div_v \left ( \mc{J}[\tilde m, D_v \tilde u ]  \right ) \\
m \rvert_{t=0} = m^0,
\end{cases}
\ee
where $m^0 \in H^s_z$, $\tilde m \in X^s$ and $\tilde u \in X^s$ are assumed to be given.
Using the estimates from Proposition~\ref{prop:Kol-WP} and Corollary~\ref{cor:Kol-bwd},
we obtain the following estimates for $u$ and $m$.

\begin{lemma} \label{lem:XsMFG}
Let $s \geq s_\ast +1$ and let $m^0 \in H^s_z$, $\tilde m \in X^s$ and $\tilde u \in X^s$.
Let $\mc{H}, \mc{J}$ and $\mc{G}$ satisfy Assumption~\ref{hyp:data-gen}.

Then there exists $u \in X^s$ satisfing \eqref{eq:value-gen} and $m \in X^s$ satisfying \eqref{eq:cty-gen} in the sense of distributions. Moreover the following estimate holds:
\be
\| m \|_{X^s}^2 + \| u \|_{X^s}^2 \leq 2 \| m^0 \|_{H^s_z}^2 + 2 \delta^2 \mc{K} (T, \| \tilde m \|_{X^s}^2) + 4 \e^2 (1+ T) \mc{F} \left ( T, (1+T) \| (\tilde m, \tilde u )\|_{X^s}^2 \right ) .
\ee

\end{lemma}
\begin{proof}
Since $m \in X^s$ satisfies \eqref{eq:cty-gen} in the sense of distributions,
by Proposition~\ref{prop:Kol-WP} with $h_1 = 0$ and $h_2 = - \e \mc{J}[\tilde m, D_v \tilde u ]$,
\be
\| m \|_{X^s}^2 \leq 2 \| m^0 \|_{H^s_z}^2 + 4 \e^2 \int_0^T \left \| \mc{J} [ \tilde m, D_v \tilde u ] \right \|_{\Gamma^s(t)}^2 \dd t .
\ee
Since $u \in X^s$ satisfies \eqref{eq:value-gen} in the sense of distributions, by Corollary~\ref{cor:Kol-bwd} with the choice of $h_1 = - \e \mc{H}[\tilde m, D_v \tilde u ] $ and $h_2 = 0$,
\be
\| u \|_{X^s}^2 \leq 2 \delta^2 \| \mc{G} [\tilde m(T, \cdot)] \|_{\Gamma^s_z(T)}^2 + 4 \e^2 T \int_0^T \left \| \mc{H} [ \tilde m, D_v \tilde u ] \right \|_{\Gamma^s(t)}^2 \dd t .
\ee
By Assumption~\ref{hyp:data-gen}(\ref{hyp:H-gen}),
\be
\int_0^T T \left \| \mc{H} [ \tilde m, D_v \tilde u ] \right \|_{\Gamma^s(t)}^2 + \left \| \mc{J} [ \tilde m, D_v \tilde u ] \right \|_{\Gamma^s(t)}^2 \dd t \leq (1+T)  \mc{F} \left (T, \| (\tilde m,D_v \tilde u) \|_{L^2_t \Gamma^s_z \cap L^\infty_t \Gamma^{s-1}_z }^2 \right ) .
\ee
By Assumption~\ref{hyp:data-gen}(\ref{hyp:G-gen}),
\be
\| \mc{G} [\tilde m(T, \cdot)]  \|_{\Gamma^s_z(T)}^2 \leq \mc{K} \left ( T,  \| \tilde m(T, \cdot) \|_{\Gamma^s_z(T)}^2 \right ) .
\ee
Finally, we note that
\be
\| D_v \tilde u \|_{L^2_t \Gamma^s_z \cap L^\infty_t \Gamma^{s-1}_z }^2 \leq \| \tilde u \|_{X^s}^2
\ee
and
\be
\| \tilde m \|_{L^2_t \Gamma^s_z \cap L^\infty_t \Gamma^{s-1}_z }^2 \leq (1+T) \| \tilde m \|_{X^s}^2 , \qquad  \| \tilde m(T, \cdot) \|_{\Gamma^s_z(T)}^2 \leq \| \tilde m \|_{X^s}^2 .
\ee
Substituting into the above estimates and summing completes the proof.
\end{proof}

We next consider differences between solutions with different forcing terms. For $i=1,2$, let $(\tilde m_i, \tilde u_i) \in X^s \times X^s$ and consider $( m_i, u_i) \in X^s \times X^s$ the distributional solutions of 
\be \label{eq:value-gen-2}
\begin{cases}
- \partial_t u_i - v \cdot D_x u_i  - \frac{1}{2} \Delta_v u_i  = \e \mc{H}[\tilde m_i , D_v \tilde u_i ] \\
u_i\rvert_{t=T} = \delta \mc{G} [  \tilde m_i (T, \cdot) ]
\end{cases}
\ee
and
\be \label{eq:cty-gen-2}
\begin{cases}
\partial_t m_i  + v \cdot D_x m _i - \frac{1}{2} \Delta_v m_i  = - \e \div_v \left ( \mc{J}[\tilde m_i, D_v \tilde u_i ]  \right ) \\
m_i \rvert_{t=0} = m^0 .
\end{cases}
\ee

\begin{lemma} \label{lem:XsMFG-difference}
Let $( m_i, u_i) , (\tilde m_i, \tilde u_i) \in X^s \times X^s$ be as above. Then
\begin{multline}
\| (m_1 - m_2,  u_1 - u_2 ) \|_{X^s}^2 \leq \| (\tilde m_1 - \tilde m_2,  \tilde u_1 - \tilde u_2 ) \|_{X^s}^2 \\
\times \left [2 \delta^2\mc{K} \left ( T, \max_i \|  \tilde m_i  \|_{X^s}^2 \right ) + 4 \e^2 (1+T)^2 \, \mc{F} \left (T, (1+T) \max_i \| (\tilde m_i, \tilde u_i ) \|_{X^s}^2 \right )  \right ] .
\end{multline}

\end{lemma}
\begin{proof}
The difference $m_1 - m_2 \in X^s$ satisfies
\be
(\partial_t  + v \cdot D_x- \frac{1}{2} \Delta_v) (m_1 - m_2)  = - \e \div_v \left ( \mc{J}[\tilde m_1, D_v \tilde u_1 ] -  \mc{J}[\tilde m_2, D_v \tilde u_2 ]  \right ) ; \quad (m_1 - m_2) \vert _{t=0} \equiv 0 .
\ee
By Proposition~\ref{prop:Kol-WP} and Assumption~\ref{hyp:data-gen}(\ref{hyp:H-gen}),
\begin{align}
\| m_1 &- m_2 \|_{X^s}^2  \leq 4 \e^2 \int_0^T \left \|  \mc{J}[\tilde m_1, D_v \tilde u_1 ] -  \mc{J}[\tilde m_2, D_v \tilde u_2 ] \right \|_{\Gamma^s(t)}^2 \dd t \\
& \leq 4 \e^2 \| (\tilde m_1 - \tilde m_2, D_v \tilde u_1 - D_v \tilde u_2 ) \|_{L^2_t \Gamma^s_z \cap L^\infty_t \Gamma^{s-1}_z}^2 \, \mc{F} \left (T, \max_i \| (\tilde m_i, D_v \tilde u_i ) \|_{L^2_t \Gamma^s_z \cap L^\infty_t \Gamma^{s-1}_z }^2 \right )  \\
& \leq 4 \e^2 (1+T) \| (\tilde m_1 - \tilde m_2, D_v \tilde u_1 - D_v \tilde u_2 ) \|_{X^s}^2  \, \mc{F} \left (T, (1+T) \max_i \| (\tilde m_i, D_v \tilde u_i ) \|_{X^s }^2 \right ).
\end{align}
The difference $u_1 - u_2 \in X^s$ satisfies
\be
\begin{cases}
- (\partial_t  + v \cdot D_x + \frac{1}{2} \Delta_v) (u_1 - u_2)  =   \e \left ( \mc{H}[\tilde m_1, D_v \tilde u_1 ] -  \mc{H}[\tilde m_2, D_v \tilde u_2 ]  \right ) ;  \\
 (u_1 - u_2) \vert _{t=0} = \mc{G}[ \tilde m_1 (T, \cdot)  ] - \mc{G}[ \tilde m_2 (T, \cdot)  ] .
 \end{cases}
\ee
By Corollary~\ref{cor:Kol-bwd} and Assumption~\ref{hyp:data-gen},
\begin{align}
\| u_1 &- u_2 \|_{X^s}^2  \leq 2 \delta^2 \left \| \mc{G}[ \tilde m_1 (T, \cdot)  ] - \mc{G}[ \tilde m_2 (T, \cdot)  ] \right \|_{\Gamma^s_z (T)}^2\\
&  + 4 \e^2 T \int_0^T \left \|  \mc{H}[\tilde m_1, D_v \tilde u_1 ] -  \mc{H}[\tilde m_2, D_v \tilde u_2 ] \right \|_{\Gamma^s(t)}^2 \dd t \\
& \leq 2 \delta^2 \|  \tilde m_1 (T, \cdot) -  \tilde m_2 (T, \cdot) \|_{\Gamma^s_z (T)}^2 \: \mc{K} \left ( T, \max_i \|  \tilde m_i (T, \cdot) \|_{\Gamma^s_z(T)}^2 \right ) \\
& \; + 4 \e^2 T \| (\tilde m_1 - \tilde m_2, D_v \tilde u_1 - D_v \tilde u_2 ) \|_{L^2_t \Gamma^s_z \cap L^\infty_t \Gamma^{s-1}_z}^2 \, \mc{F} \left (T, \max_i \| (\tilde m_i, D_v \tilde u_i ) \|_{L^2_t \Gamma^s_z \cap L^\infty_t \Gamma^{s-1}_z }^2 \right ) \\
& \leq 2 \delta^2 \|  \tilde m_1 -  \tilde m_2 \|_{X^s}^2 \: \mc{K} \left ( T, \max_i \|  \tilde m_i  \|_{X^s}^2 \right )  \\
& \; + 4 \e^2 T (1+T) \| (\tilde m_1 - \tilde m_2,  \tilde u_1 - \tilde u_2 ) \|_{X^s}^2 \, \mc{F} \left (T, (1+T) \max_i \| (\tilde m_i, \tilde u_i ) \|_{X^s}^2 \right ) .
\end{align}
Summing completes the proof.
\end{proof}

\section{Local Well-posedness}
\label{sec:FixedPoint}

We prove local well-posedness by a fixed point argument.

\begin{definition} Let $m^0 \in H^s_z$ be a given probability density function. Let $\mc{H}, \mc{J}, \mc{G}$ satisfy Assumption~\ref{hyp:data-gen}.
The map $\mc{M} : X^s \times X^s \to X^s \times X^s$ is defined for all $(\tilde m, \tilde u) \in X^s \times X^s$ by
\be
\mc{M} (\tilde m, \tilde u) = (m, u),
\ee
where $m$ is the unique $X^s$ solution of \eqref{eq:cty-gen} and $u$ is the unique $X^s$ solution of \eqref{eq:value-gen}.
\end{definition}

Note that Proposition~\ref{prop:Kol-WP} ensures that $\mc{M}$ is well defined.
We will show that, for suitable choices of the parameters $\e, \delta$ and $T$, $\mc{M}$ is a contraction on a closed ball in $X^s \times X^s$. 
It will then follow from the Banach fixed point theorem that $\mc{M}$ has a unique fixed point on this set.  This is the content of the next lemma.
Recall the functions $\mathcal{F}$ and $\mathcal{K}$ from Assumption \ref{hyp:data-gen}.  

\begin{lemma} \label{lem:M}

Suppose that there exists $L_\ast > 0$ such that
\be \label{hyp:Last}
2 \delta^2 \mc{K} \left ( T, L_\ast \right ) + 4 \e^2 (1+ T) \, \mc{F} \left ( T, (1+T) L_\ast \right ) < L_\ast
\ee
Define
\be \label{def:Kast}
K_\ast : = \sup \left \{ K > 0 : 2 \delta^2 \mc{K} \left ( T, L_\ast + K \right ) + 4 \e^2 (1+ T) \, \mc{F} \left ( T, (1+T) (L_\ast + K) \right ) < L_\ast . \right \},
\ee
with the convention that $K_\ast = + \infty$ if the set is unbounded.

Then, for any $m^0 \in H^{s}(\dom)$ such that
\be
\| m^0 \|_{H^{s}_{z}}^2 \leq \frac{K}{2} < \frac{K_\ast}{2},
\ee
$\mc{M}$ maps the set
\be \label{def:Y}
Y : = \left \{ (m,u) \in X^s \times X^s \, \vert \,  \| (m, u) \|_{X^s}^2 \leq K + L_\ast \right \}
\ee
into itself, and is Lipschitz on $Y$ with Lipschitz constant $\sqrt{(1+T) L_\ast}$.
\end{lemma}
\begin{rmk} \label{rmk:condition}
The condition \eqref{hyp:Last} can be satisfied for any $L_\ast$ and $T$ by taking $\delta, \e$ small enough.
Likewise, $K_{\ast}$ can be taken as large as desired by taking $\delta, \e$ small enough.
\end{rmk}

\begin{proof}[Proof of Lemma \ref{lem:M}]
By Lemma~\ref{lem:XsMFG}, if $(m,u) \in X^s \times X^s$ then
\be
\| \mc{M} (m,u) \|^2_{X^s}  \leq 2 \| m^0 \|_{H^s_z}^2 + 2 \delta^2 \mc{K} (T, \|  (m,u) \|^2_{X^s} ) + 4 \e^2 (1+ T) \mc{F} \left ( T, (1+T) \|  (m,u) \|^2_{X^s} \right )  .
\ee
Hence, if $ \| (m, u) \|_{X^s}^2 \leq K + L_\ast$, then
\be
\| \mc{M} (m,u) \|^2_{X^s}  \leq K + 2 \delta^2 \mc{K} (T, K + L_\ast ) + 4 \e^2 (1+ T) \mc{F} \left ( T, (1+T)(K + L_\ast) \right )  .
\ee
Thus, since $K < K_\ast$,
\be
\| \mc{M} (m,u) \|^2_{X^s}  \leq K + L_\ast ,
\ee
which shows that $\mc{M}$ maps $Y$ into $Y$.

Next, for $i=1,2$ consider $(m_i , u_i ) \in Y$. By Lemma~\ref{lem:XsMFG-difference},
\begin{multline}
\| \mc{M}(m_1, u_1) - \mc{M}(m_2, u_2) \|_{X^s}^2 \leq \| ( m_1,  u_1) - (m_2, u_2 ) \|_{X^s}^2 \\
\times \left (2 \delta^2\mc{K} \left ( T, \max_i \|  m_i \|_{X^s}^2 \right ) + 4 \e^2 (1+T)^2 \, \mc{F} \left (T, (1+T) \max_i \| ( m_i, u_i ) \|_{X^s}^2 \right )  \right ) .
\end{multline}
Since each $(m_i , u_i ) \in Y$, we may estimate this by
\begin{multline}
\| \mc{M}(m_1, u_1) - \mc{M}(m_2, u_2) \|_{X^s}^2 \leq \| ( m_1,  u_1) - (m_2, u_2 ) \|_{X^s}^2 \\
\times \left (2 \delta^2\mc{K} \left ( T, K + L_\ast \right ) + 4 \e^2 (1+T)^2 \, \mc{F} \left (T, (1+T) (K + L_\ast)\right )  \right ) .
\end{multline}
Since $K < K_\ast$,
\be
2 \delta^2\mc{K} \left ( T, K + L_\ast \right ) + 4 \e^2 (1+T)^2 \, \mc{F} \left (T, (1+T) (K + L_\ast)\right ) \leq (1+T) L_\ast .
\ee
Thus
\be
\| \mc{M}(m_1, u_1) - \mc{M}(m_2, u_2) \|_{X^s}^2 \leq(1+T) L_\ast   \| ( m_1,  u_1) - (m_2, u_2 ) \|_{X^s}^2 
\ee
which completes the proof.
\end{proof}

We are ready to state and prove the main theorem of our paper.

\begin{thm}\label{thm:main_proved}
Let $s \geq s_\ast + 1$.
Assume that $\mc{H}, \mc{J}$ and $\mc{G}$ satisfy Assumption~\ref{hyp:data-gen}.
Suppose that $\delta, \e, T > 0$ are such that
there exists $L_\ast < \frac{1}{1 + T}$ satisfying \eqref{hyp:Last}.
Then, for any $m^0 \in H^s_z$ such that $\| m^0 \|_{H^s_z}^2 < K_\ast /2$, there exists a unique distributional solution $(m,u) \in X^s \times X^s$ of the Mean Field Games system \eqref{eq:MFG-gen} such that $\| (m,u) \|_{X^s} < L_\ast + K_\ast$.

\end{thm}
\begin{proof}
Let $K>0$ be such that $2 \| m^0 \|_{H^s_z} \leq K < K_\ast$, and let $Y$ be defined by \eqref{def:Y}.
Since $(1+T)L_\ast < 1$, by Lemma~\ref{lem:M} $\mc{M}$ is a contraction on $Y$, which is a complete metric space (and non-empty since $(0,0) \in \mc{M}$).
Hence, by the Banach fixed point theorem, $\mc{M}$ has a unique fixed point $(m,u) \in Y$. $\mc{M}(m,u) = (m,u)$ is precisely the statement that $(m,u)$ satisfies the Mean Field Games system \eqref{eq:MFG-gen}.
\end{proof}

\begin{rmk}\label{probabilityRemark}
The solution $m$ is a probability distribution at each time as long as the initial data $m^{0}$ is.  This is because the equation for $m$
is positivity-preserving and also preserves the mean, as long as solutions have some slight regularity.  The $H^{s}$ regularity of the
theorem is certainly sufficient.
\end{rmk}

As a consequence of Theorem~\ref{thm:main_proved} and Proposition~\ref{prop:HRegularity}, in the local case we have the following theorem.

\begin{thm}\label{finalTheorem}
Let $s \geq s_\ast + 1$. Let $H, g$ satisfy Assumptions~\ref{hyp:main} and \ref{hyp:terminal}.
Let $m^0 \in H^s_z$ and $T>0$ be given. 
There exist positive parameters $(\e_\ast, \delta_\ast)$ such that, for all $\e < \e_\ast$, $\delta < \delta_\ast$, there exists a distributional solution $(u,m) \in X^s \times X^s$ of \eqref{eq:MFG}, which is the unique solution such that $\|(u, m)\|_{X^s} \leq Q_\ast$, for some constant $Q_\ast > 2 \| m_0 \|_{H^s_z}$.

Furthermore, if $s \geq s_\ast + 2$, the solution is classical.
\end{thm}
\begin{proof}
Choose constants $K_\ast > 2 \| m_0 \|_{H^s_z}$ and $L_\ast < (1+T)^{-1}$. As explained in Remark~\ref{rmk:condition}, the condition \eqref{hyp:Last} can be satisfied for these values of $L_\ast$, $T$ and $K_\ast$ by taking $\e_\ast$ and $ \delta_\ast$ small enough. The existence and uniqueness statement then follows from Theorem~\ref{thm:main_proved}.

It remains to prove only the final assertion. If $s \geq s_\ast + 2$, then $u, m \in C^0_t C^{2,\alpha}_z$. Thus
\begin{align}
& \partial_t u = - v \cdot D_x u - \frac{1}{2} \Delta_v u + \e H(t,z, m, D_v u) \in C^0_t C^{0,\alpha}_z \\
& \partial_t m = - v \cdot D_x m + \frac{1}{2} \Delta_v m - \e \div_v \left ( m D_p H(t,z, m(t,z), D_v u(t,z)) \right ) \in C^0_t C^{0,\alpha}_z,
\end{align}
and hence the solution is classical.
\end{proof}

\begin{rmk} If instead one wished to fix $\e$ and $\delta$ rather than take them as small as necessary for Theorem \ref{finalTheorem},
it is in some cases still possible to find existence of a solution.  As long as the Hamiltonian is such that the associated functions 
$\mathcal{F}$ and $\mathcal{K}$ admit $T>0$ and $L_{*}>0$ such that \eqref{hyp:Last} is satisfied, and if $L_{\ast}<(1+T)^{-1},$
then there exists $K_{*}>0$ such that
for any $m_{0}$ with $\|m_{0}\|_{H_{z}^{s}}^{2}<\frac{K_{*}}{2},$ we have demonstrated that there is a solution of \eqref{eq:MFG}.
\end{rmk}

\appendix 

\section{Approximation argument for a key estimate}\label{appendixSection}

We give the approximation argument to show that, if $w \in X^0$ satisfies
\be \label{eq:Kol-app}
\left(\partial_t + v \cdot D_x -  \frac{1}{2}   \Delta_v\right) w = h_1 + \div_v h_2 
\ee
in the sense of distributions $\mc{D}'((0,T) \times \dom )$, where $h_1, h_2 \in L^2_{t,x,v}$,
then for all $t \in [0,T]$,
\be
\frac{1}{2} \| w(t) \|_{L^2_{z}}^2 + \frac{1}{2}  \int_0^t \| D_v w(\tau) \|_{L^2_{z}}^2 \dd \tau  = \frac{1}{2} \| w(0) \|_{L^2_{z}}^2 + \int_0^t \int_{\dom}  h_1   w \dd z - \int_0^t \int_{\dom}  h_2 D_v w \dd z\dd \tau .
\ee

Recall that $w$ satisfies \eqref{eq:Kol-app} in $\mc{D}'\left ( (0,T) \times \dom \right )$ if and only if, for all test functions $\phi \in C^\infty_c \left ( (0,T) \times \dom \right)$, the following equality holds:
\be
 - \int_0^T \int_{\dom} w \left (\partial_t + v \cdot \nabla_x +  \frac{1}{2}  \Delta_v \right ) \phi \dd z \dd t = \int_0^t \int_{\dom} h_1 \phi - h_2 \cdot \nabla_v \phi \dd z \dd t .
\ee
Since $\nabla_v w \in L^{2}_{t,z}$ by assumption, it is equivalent to require that
\begin{multline} \label{eq:Kol-weak}
 - \int_0^T \int_{\dom} w \left (\partial_t + v \cdot \nabla_x \right )\phi \dd z \dd t  +  \frac{1}{2}  \int_0^T \int_{\dom} \nabla_v w \cdot \nabla_v \phi \dd z \dd t \\
 = \int_0^t \int_{\dom} h_1 \phi - h_2 \cdot \nabla_v \phi \dd z \dd t .
\end{multline}

We now construct a sequence of smooth test functions that approximates $w.$
Fix a smooth, non-negative, radially symmetric function $\psi \in C^\infty_c(\R \times \dom)$ and define for $\delta > 0$
\be
\psi_\delta(t,x,v) : = \delta^{-(1 + 2d)} \psi \left (\frac{t}{\delta},\frac{x}{\delta},\frac{v}{\delta} \right) .
\ee
We localise separately in the $t$ and $(x,v)$ variables. In $(x,v)$, for each $R>0$ let $\zeta_R \in C^\infty_c(\dom)$ be a non-negative function satisfying
\be
\zeta_R(z) = \begin{cases}
1, & |z| \leq R, \\
0, & |z| > 2R,
\end{cases} \quad \| \nabla \zeta_R \|_{L^\infty} \leq \frac{C}{R} .
\ee
For the time variable, we fix a time $t \in (0,T]$. Then, for each $0 < \tau < t/2$, let $\eta_\tau$ be a smooth non-negative function of $s$ such that
\be
\eta(s) = \begin{cases}
1, & \tau \leq s \leq t -\tau \\
0, & s \leq \frac{\tau}{2}, \; s \geq t - \frac{\tau}{2}
\end{cases} \quad \| \eta_\tau ' \|_{L^\infty} \leq \frac{C}{\tau} .
\ee

Now define the test function $\phi : = \eta_\tau \zeta_R (\psi_\delta \ast (\eta_\tau \zeta_R w))$. We substitute this choice of $\phi$ into the weak form \eqref{eq:Kol-weak} to find (dropping subscripts for ease of notation) that
\begin{multline} \label{eq:Kol-weak-tested}
 - \int_0^t \int_{\dom} ( (\partial_t \eta)  \zeta w) \, \psi \ast (\eta \zeta w) \dd z \dd s +  \frac{1}{2}  \int_0^t \int_{\dom} (\eta  \zeta \nabla_v w)  \cdot \psi \ast (\eta \zeta \nabla_v w) \dd z \dd s   \\
  - \int_0^t \int_{\dom} (\eta \zeta w) \left ((\partial_t + v \cdot \nabla_x)\psi \right) \ast (\eta \zeta w) \dd z \dd s  \\
=   \int_0^t \int_{\dom} ( \eta \zeta h_1) \psi \ast (\eta \zeta w)   - \eta \zeta h_2 \cdot [ \psi \ast (\eta \zeta \nabla_v w) ] \dd z \dd s \\
  +  \int_0^t \int_{\dom} \eta \left [ \left ( v \cdot \nabla_x \zeta \right )  w  - \nabla_v \zeta \cdot \left(h_2 +  \frac{1}{2}   \nabla_v w \right) \right ] \, \psi \ast (\eta \zeta w) \dd z \dd s   \\
  -  \int_0^t \int_{\dom} \eta \zeta \left( h_2 + \frac{1}{2}   \nabla_v w \right) \cdot  \psi \ast  (\eta \nabla_v  \zeta w) \dd z \dd s .
\end{multline}
Since $\psi$ is radially symmetric,
\be
\partial_t \psi(s,x,v) = - \partial_t \psi(-s,x,v), \qquad \nabla_x \psi(s,x,v) = - \nabla_x \psi(s, - x,v).
\ee
It follows that
\begin{align*}
\int_0^t \int_{\dom} (\eta \zeta w) \left ((\partial_t + v \cdot \nabla_x)\psi \right)& \ast (\eta \zeta w) \dd z \dd s\\ 
&= - \int_0^t \int_{\dom} (\eta \zeta w) \left ((\partial_t + v \cdot \nabla_x)\psi \right) \ast (\eta \zeta w) \dd z \dd s  = 0 .
\end{align*}

We now take the limit $\delta \to 0$. Since $w, \nabla_v w \in L^2_{t,z}$, the same is true of the functions $\eta \zeta w, \eta \nabla \zeta w, \eta \zeta \nabla_v w \in L^2_{t,z}$. Then, by standard results on the continuity of mollification (continuity of translations) in $L^2$,
\be
\lim_{\delta \to 0} \| \psi_\delta \ast (\eta \zeta w) - \eta \zeta w \|_{L^2_{t,z}} + \| \psi_\delta \ast (\eta \nabla \zeta w) - \eta \nabla \zeta w \|_{L^2_{t,z}} + \| \psi_\delta \ast (\eta \zeta \nabla_v w) - \eta \zeta \nabla_v w \|_{L^2_{t,z}} = 0 . 
\ee
Since also $(\partial_t \eta)  \zeta w \in L^2_{t,z}$, we may pass to the limit in equation \eqref{eq:Kol-weak-tested} to find that
\begin{multline} \label{eq:weak-energy-eta-zeta}
 - \frac{1}{2} \int_0^t \int_{\dom} \partial_t (\eta^2) \zeta^2 |w|^2 \dd z \dd s + \frac{1}{2}  \int_0^t \int_{\dom} \eta^2  \zeta^2 |\nabla_v w|^2 \dd z \dd s   \\
=   \int_0^t \int_{\dom}  \eta^2 \zeta^2 [ h_1  w  -  h_2 \cdot  \nabla_v w] \dd z \dd s \\
  +  \int_0^t \int_{\dom} \eta^2 \zeta \left ( v \cdot \nabla_x \zeta \right ) w^2  \dd z \dd s -  \int_0^t \int_{\dom} \eta^2 \zeta \nabla_v  \zeta \cdot ( \nabla_v w + 2  h_2 ) w \dd z \dd s .
\end{multline}

Next, we take the limit $R \to +\infty$. The functions $\zeta^2, \zeta \left ( v \cdot \nabla_x \zeta \right )$ and $\zeta \nabla_v  \zeta$ are all bounded uniformly in $R$ and converge pointwise as $R \to + \infty$. Then, since $w, \nabla_v w, h_1$ and $h_2$ are all $L^2_{t,z}$ functions, we may pass to the limit in the integrals in \eqref{eq:weak-energy-eta-zeta} by the dominated convergence theorem. Since
\be
\lim_{R \to +\infty} \zeta^2 = 1, \quad \lim_{R \to +\infty} \zeta \left ( v \cdot \nabla_x \zeta \right ) = 0,  \quad \lim_{R \to +\infty} \zeta \nabla_v  \zeta = 0,
\ee
we deduce that
\be \label{eq:weak-energy-eta}
 - \frac{1}{2} \int_0^t \int_{\dom} \partial_t (\eta^2) |w|^2 \dd z \dd s + \frac{1}{2} \int_0^t \int_{\dom} \eta^2 |\nabla_v w|^2 \dd z \dd s 
=   \int_0^t \int_{\dom}  \eta^2 [ h_1  w  -  h_2 \cdot  \nabla_v w] \dd z \dd s .
\ee

Finally, we take the limit $\tau \to 0$. Since $\lim_{\tau \to 0} \eta^2 = 1$ pointwise, the convergence of the two terms involving $\eta^2$ follows by dominated convergence as before. It remains only to consider the term
\be
 - \frac{1}{2} \int_0^t \int_{\dom} \partial_t (\eta^2) |w|^2 \dd z \dd s =  - \frac{1}{2} \int_0^\tau \int_{\dom} \partial_t (\eta^2) |w|^2 \dd z \dd s - \frac{1}{2} \int_{t-\tau}^t \int_{\dom} \partial_t (\eta^2) |w|^2 \dd z \dd s .
\ee
Now recall that $s \mapsto w(s,\cdot)$ is a continuous trajectory in $L^2(\dom)$. We write, for example,
\begin{align*}
 - \frac{1}{2} \int_0^\tau \int_{\dom} \partial_t (\eta^2) |w|^2 \dd z \dd s& =  - \frac{1}{2} \| w(0)\|_{L^2_z}^2 \int_0^\tau \partial_t (\eta^2)  \dd s\\ 
 &+ \frac{1}{2} \int_0^\tau \partial_t (\eta^2) (\| w(0)\|_{L^2_z}^2 - \| w(s)\|_{L^2_z}^2) \dd s .
\end{align*}
We evaluate the integral 
\be
 \int_0^\tau \partial_t (\eta^2)  \dd s = \eta(\tau)^2 - \eta(0)^2 = 1 - 0 = 1.
 \ee
 For the remaining term, since $|\frac{1}{2} (\eta^2)'| \leq \frac{C}{\tau}$,
 \be
 \left | \frac{1}{2} \int_0^\tau \partial_t (\eta^2) (\| w(0)\|_{L^2_z}^2 - \| w(\tau)\|_{L^2_z}^2) \dd s \right | \leq \frac{C}{\tau} \, \tau \, \sup_{s \in [0,\tau]} (\| w(0)\|_{L^2_z}^2 - \| w(s)\|_{L^2_z}^2) .
 \ee
The right hand side converges to zero as $\tau$ tends to zero by $L^2(\dom)$ continuity of $w$. Thus
\be
\lim_{\tau \to 0}  - \frac{1}{2} \int_0^\tau \int_{\dom} \partial_t (\eta^2) |w|^2 \dd z \dd s = - \frac{1}{2} \| w(0)\|_{L^2_z}^2 .
\ee
Similarly,
\be
\lim_{\tau \to 0}  - \frac{1}{2} \int_0^\tau \int_{\dom} \partial_t (\eta^2) |w|^2 \dd z \dd s = - \frac{1}{2} \| w(t)\|_{L^2_z}^2 \left ( \int_{t-\tau}^t \partial_t (\eta^2)  \dd s \right ) = \frac{1}{2} \| w(t)\|_{L^2_z}^2 .
\ee

We conclude that, for all $t \in [0,T]$,
\be \label{eq:app-weak-energy}
\frac{1}{2} \| w(t)\|_{L^2_z}^2 + \frac{1}{2} \int_0^t \int_{\dom} |\nabla_v w|^2 \dd z \dd s 
=  \frac{1}{2} \| w(0)\|_{L^2_z}^2 + \int_0^t \int_{\dom}  h_1  w  -  h_2 \cdot  \nabla_v w \dd z \dd s .
\ee

\section*{Acknowledgements}
DMA is grateful to the National Science Foundation for support through grant DMS-2307638. MGP is grateful for support from the Additional Funding Programme for Mathematical Sciences, delivered by EPSRC (EP/V521917/1) and the Heilbronn Institute for Mathematical Research. ARM has been partially supported by the EPSRC New Investigator Award ``Mean Field Games and Master equations'' under award no. EP/X020320/1 and by the King Abdullah University of Science and Technology Research Funding (KRF) under award no. ORA-2021-CRG10-4674.2.

\bibliography{kMFGbib}
\bibliographystyle{abbrv}

\end{document}